\documentclass[12pt]{article}
\usepackage{amstext}
\usepackage{amssymb}
\usepackage{latexsym}
\usepackage{amsmath}
\usepackage{amscd}
\usepackage{amsthm}
\usepackage{graphicx}

\newtheorem{theorem}{Theorem}[section]
\newtheorem{proposition}[theorem]{Proposition}
\newtheorem{lemma}[theorem]{Lemma}
\newtheorem{corollary}[theorem]{Corollary}

\newtheorem{remark}[theorem]{Remark}
\newtheorem{definition}[theorem]{Definition}
\newtheorem{notation}[theorem]{Notation}

\newcommand{\BP}{\mathbf{BP}}

\newcommand{\fpe}{\mathcal{F}^{\rm pe}}

\newcommand{\F}{\mathbb{F}}
\newcommand{\N}{\mathbb{N}}
\newcommand{\Z}{\mathbb{Z}}

\newcommand{\R}{\mathbb{R}}
\newcommand{\C}{\mathbb{C}}

\newcommand{\e}{{\rm e}}

\newcommand{\ord}{{\rm ord}}

\newcommand{\vol}{{\rm Vol}}
\newcommand{\pos}{{\rm pos}}

\newcommand{\Acal}{\mathcal{A}}
\newcommand{\Bcal}{\mathcal{B}}
\newcommand{\Ccal}{\mathcal{C}}
\newcommand{\Dcal}{\mathcal{D}}
\newcommand{\Gcal}{\mathcal{G}}

\newcommand{\Lcal}{\mathcal{L}}

\newcommand{\Ncal}{\mathcal{N}}
\newcommand{\Pcal}{\mathcal{P}}
\newcommand{\Qcal}{\mathcal{Q}}

\newcommand{\Fcal}{\mathcal{F}}

\renewcommand{\Pcal}{\mathcal{P}}
\newcommand{\Ucal}{\mathcal{U}}
\newcommand{\Vcal}{\mathcal{V}}

\newcommand{\Zcal}{\mathcal{Z}}

\newcommand{\Cfrak}{\mathfrak{C}}
\newcommand{\Dfrak}{\mathfrak{D}}

\newcommand{\Mfrak}{\mathfrak{M}}
\newcommand{\Nfrak}{\mathfrak{N}}
\newcommand{\pfrak}{\mathfrak{p}}
\newcommand{\Gfrak}{\mathfrak{G}}

\newcommand{\Qfrak}{\mathfrak{Q}}
\newcommand{\Ufrak}{\mathfrak{U}}
\newcommand{\Vfrak}{\mathfrak{V}}
\newcommand{\Wfrak}{\mathfrak{W}}
\newcommand{\Zfrak}{\mathfrak{Z}}

\title{Uniform Definability and Undecidability in Classes of Structures}

\author{Hector Pasten\\
Universidad de Concepci\'on\\
and\\
Thanases Pheidas\\
University of Crete\\
and\\
Xavier Vidaux\\
Universidad de Concepci\'on}
\date{}

\begin{document}

\date{}
\maketitle
\tableofcontents

\newpage

\begin{abstract} We present a concept of \emph{uniform encodability of theories} and develop tools related to this concept. As an application we obtain general undecidability results which are uniform for large families of structures. 
In the way, we define uniformly in the characteristic the equivalence relation ``$x\sim y$ if and only if $x$ is an iterate of $y$ throuh the Frobenius map, or vice versa'' in large classes of function fields and in polynomial rings. 
\end{abstract}

\section{Introduction}

It is well known that a system of Diophantine equations has a complex solution if and only if it has a solution modulo infinitely many primes (see \cite{Navarro}). Since there is an algorithm to solve the former problem, there is also an algorithm to decide whether an arbitrary system of diophantine equations has a solution in the finite field $\F_p$ for infinitely many primes $p$. In this work we show that the situation is completely different if we replace the fields $\F_p$ by rings of functions of positive characteristic and consider analogous diophantine problems. For example, we show that the following problems are undecidable: decide whether or not a system of diophantine equations together with conditions of the form ``$x$ is non constant'', for some of the unknowns $x$, has a solution in $\F_p[z]$ for
\begin{enumerate}
\item some odd prime $p$,
\item all odd primes $p$,
\item infinitely many odd primes $p$,
\item all but possibly a finite number of odd primes $p$,
\item all primes $p$ of the form $6k+5$, etc.
\end{enumerate}

Indeed we prove such very general \emph{uniform} undecidability results for large classes of subrings of function fields of curves (of large enough characteristic) - for example, for the class of all polynomial rings of odd positive characteristic (see Corollary \ref{Teoremon3}). In this work we will focus on existential theories, but indeed many of the tools that we develop here can be used more generally for theories of formulas of a given hierarchy. 

There seems to be rather few results of this kind in the bibliography, but there are several results on \emph{asymptotic (un)decidability}: given a class of structures, to decide whether or not a given formula is true for all but finitely many of them. For example, in \cite{ChatzHrush}, Chatzidakis and Hrushovski prove that a certain class of differential fields, each of them separately having a decidable theory, has an asymptotic undecidable theory. On the other hand, Hrushovski \cite{Hrushovski} and Macintyre \cite{Macintyre} (independently) show that the class of algebraically closed fields in positive characteristic, together with the Frobenius map, is asymptotically decidable. Other results of the same flavour can be found in \cite{AxKochen12,AxKochen3,Ax,Rumely}.

In this work, we will be interested in positive existential theories, because of the obvious connection with Hilbert's tenth problem, but the general method that we develop is straightforward adaptable to decidability questions about full theories. 

On the way, we define positive existentially the relation ``$y$ is a $p^s$-th power of $x$'' in a class of algebraic function fields whose fields of constants are algebraic over $\F_p$, for $p$ large enough with respect to the genus.

Hilbert's tenth problem (the tenth in the famous list that Hilbert gave at the International Conference of Mathematicians in Sorbonne, in 1900) was: \\

\noindent\emph{to find a process according to which one can determine in a finite number of steps whether a polynomial equation with integer coefficients has or does not have integer solutions.}\\

The problem was answered in 1971 when Y. Matiyasevich, based on work of J. Robinson, M. Davis and H. Putnam, proved that no such `process' (in modern terminology: algorithm) exists - and all this was built on the foundational work of K. G\"odel and A. Turing who laid the necessary foundations in Logic (see \cite{Davis} and \cite{Matiyasevich}). Later various authors asked similar questions for rings other than the integers (starting with J. Denef and L. Lipshitz). One such question is the following: What if we replace the \emph{integers} by \emph{polynomials}, say in one variable, with coefficients in a finite field $\F_q$, with $q=p^n$ elements, where $p$ is the (prime) characteristic. The problem was answered by J. Denef in \cite{Denef1} and \cite{Denef2}, negatively again. In the modern terminology of Logic the result is phrased \emph{The positive existential theory of a ring $F[z]$ of polynomials of the variable $z$ over a field $F$, in the language $\Lcal_z=\{0,1,+,\cdot,z\}$, is undecidable}. In this problem the considered polynomial equations are those with coefficients in the natural image of $\Z[z]$ in $F[z]$. 

Later, a large number of similar results (mostly of a negative nature) were established. The general flavor of these results is: if in place of $\Z$ in Hilbert's tenth problem we substitute a global ring or field, such as a ring of polynomials or rational functions (or a finite extension), all the existing results are negative (the positive existential theory is undecidable); almost always in the language $\Lcal_z$ or an extension of it by a finite list of symbols for certain elements of the structure. In contrast, in local domains, such as a field of $p$-adic numbers or power series, the results tend to be positive (decidable existential theory, even decidable first order theory). But there are many open problems, for example the question asked for $\C(z)$, the field of rational functions with complex coefficients (or coefficients in any algebraically closed field) and the field of formal power series over any reasonable field of positive characteristic (e.g. over a finite field). 

In order to state our results, we need to introduce a few notation. All languages considered will be first order languages. Also, the word \emph{class} will always refer to a non-empty class of structures over a common language. 

\begin{notation}\label{cafe}
\begin{enumerate}
\item We consider $0$ to be a natural number. 
\item All languages considered will be first order and equalitarian.
\item\label{form} If $\Lcal$ is a language, we will denote by $\Fcal_\Lcal$ (respectively $\Fcal_\Lcal^{\rm e}$, $\Fcal_\Lcal^{\rm pe}$) the set of (respectively existential, positive existential) $\Lcal$-sentences, and if $\Mfrak$ is an $\Lcal$-structure $T_\Lcal(\Mfrak)$ (respectively $T_\Lcal^{\rm e}$, $T_\Lcal^{\rm pe}$) will stand for the (respectively existential, positive existential) $\Lcal$-theory of $\Mfrak$.
\item\label{forget} If $\Ufrak$ is an $\Lcal$-structure and $X$ is a subset of $\Lcal$, we will denote by $\Ufrak_X$ the $\Lcal\smallsetminus X$-structure in which we \emph{forget} the interpretation of the symbols of $X$. If $\Ucal$ is a class of such $\Lcal$-structures, we will denote by $\Ucal_X$ the class of corresponding $\Lcal\smallsetminus X$-structures. 
\item\label{expand} If $\Ufrak$ is an $\Lcal$-structure and $X$ is a set of symbols which are not in $\Lcal$ and which have a given interpretation in $\Ufrak$, we will denote by $\Ufrak^X$ the corresponding $\Lcal\cup X$-structure. If $\Ucal$ is a class of such $\Lcal$-structures we will denote by $\Ucal^X$ the corresponding class of $\Lcal\cup X$-structures. 
\item All classes of structures are by default non-empty.
\item\label{languages} We define the following languages:
\begin{enumerate}
\item $\Lcal_A=\{0,1,+,\cdot\}$ is the language of rings;
\item $\Lcal_z=\Lcal_A\cup\{z\}$, where $z$ is a symbol of constant;
\item $\Lcal_{z,\ord}=\Lcal_z\cup\{\ord\}$, where $\ord$ is a unary predicate symbol;
\item $\Lcal_T=\Lcal_A\cup\{T\}$, where $T$ is a unary predicate symbol;
\item $\Lcal_T^*=\{0,1,+,\mid,R,T\}$, where $\mid$ and $R$ are binary relation symbols;
\item $\Lcal^{*,+}=\{0,1,+,R\}$.
\item $\Lcal^*=\{0,1,+,\pos,R\}$, where $\pos$ is a unary relation symbol interpreted in $\Z$ as:  ``$\pos(x)$ if and only if $x$ is non-negative''. We will freely write $x\geq0$ when working over this language. 
\end{enumerate}
\item\label{midp} For each prime $p$, consider the following equivalence relation $\mid_p$ over $\Z$:
$$
x\mid_p y \textrm{ if and only if there exists } s\in\Z \textrm{ such that } y=\pm xp^s.
$$
We will refer to it as \emph{$p$-divisibility} and denote its restriction to the natural numbers by the same symbol. 
\item\label{Dcal} Let $\Dfrak_p$ be the $\Lcal_T^*$-structure $(\Z;0,1,+,\mid,\mid_p,\Z\smallsetminus\{-1,0,1\})$ and 
$$
\Dcal=\{\Dfrak_p\colon p\textrm{ is prime}\}.
$$
\item\label{Ncal} Let $\Nfrak_p$ be the $\Lcal^{*,+}$-structure $(\N;0,1,+,\mid_p)$ and 
$$
\Ncal=\{\Nfrak_p\colon p\textrm{ is prime}\}.
$$
\item\label{Zcal} Let $\Zfrak_p$ be the $\Lcal^*$-structure $(\Z;0,1,+,\geq0,\mid_p)$ and 
$$
\Zcal=\{\Zfrak_p\colon p\textrm{ is prime}\}.
$$
\item\label{LzFF} All function fields will be considered as structures over $\Lcal_{z,\ord}$, where $z$ is interpreted as a local parameter at a prime divisor $\pfrak$ of the field over its field of constants, and $\ord(x)$ will be interpreted as ``the valuation of $x$ at $\pfrak$ is non-negative''. The symbol $z$ will just be the variable $z$ in the case of a rational function field $F(z)$, and in this case $\ord(x)$ will be interpreted as ``the order of $x$ at $0$ is non-negative''.
\item\label{LzRat} Any subring $B$ of a rational function field $F(z)$, whose elements are regular at $0$, will be considered as an $\Lcal_z$-structure, where the symbol $z$ is interpreted as the variable $z$, or, in the case that $B$ is a ring of polynomials $F[z]$, it will also be considered as an $\Lcal_T$-structure, where $T(x)$ will be interpreted as ``$x$ is non-constant''.
\end{enumerate}
\end{notation}

\begin{definition}
Let $\Lcal$ be a first order language, $X$ be a non-empty proper subset of $\Lcal$, and let $\Ucal$ be a class of $\Lcal$-structures. We will say that a symbol $\alpha\in X$ is \emph{uniformly $\Lcal\smallsetminus X$-definable in $\Ucal_X$} (or in $\Ucal$ if there is no ambiguity), if there exists an $\Lcal\smallsetminus X$-formula which defines the interpretation of $\alpha$ in each element of $\Ucal$. If moreover the formula is existential, respectively positive existential, then we will say \emph{uniformly existentially $\Lcal\smallsetminus X$-definable}, respectively \emph{uniformly positive existentially $\Lcal\smallsetminus X$-definable ($\Lcal\smallsetminus X$-uped)}, instead of just uniformly $\Lcal\smallsetminus X$-definable.
\end{definition}

If the symbol $\alpha$ has the same name `x' across its interpretations in elements of $\Ucal$, we will say that `x' is uniformly $\Lcal\smallsetminus\{\alpha\}$-definable. Also we may say that the family of interpretations of $\alpha$ is uniformly definable in $\Ucal$ instead of  saying that $\alpha$ is.

Let us give a few trivial examples to illustrate the definition:
\begin{itemize} 
\item With the language $\{R,\cdot\}$ and the class of all groups (where $R(x)$ is interpreted as ``$x$ is in the center'' and the symbol $\cdot$ is interpreted as the group law), 
the formula $\forall y(xy=yx)$ uniformly $\{\cdot\}$-defines $R$ in the class of all groups. 
\item With the language $\{\e,\cdot\}$ and the class of all groups (where $\e$ is interpreted as the identity element and $\cdot$ is interpreted as the group law), the formula $\exists y(x\cdot y=y)$ uniformly positive existentially $\{\cdot\}$-defines $\e$ in $\Ucal$ over the language $\{\cdot\}$. So we shall say that the identity element is $\{\cdot\}$-uped in the class of all groups. 
\end{itemize}

Another elementary example is given by the following lemma which we will prove in Section \ref{ProofCognac}.\\

\begin{lemma}\label{polne}
The relation $\ne$ is $\Lcal_z$-uped in the class of all polynomial rings over fields, where $z$ is interpreted as the variable. 
\end{lemma}

Moret-Bailly in \cite{MoretBailly} gives very general criteria for positive existential (un)definability of the relation $\ne$ in rings. 

As a non-trivial example, we prove the following proposition in Section \ref{ProofCognac}.

\begin{proposition}\label{Cognac}
Consider the language 
$$
\Lcal=\{0,+,\leq,R_2\}
$$ 
and the structures 
$$
\Cfrak_r=(\Z;0,+,\leq,P_2^r)
$$ 
where $P_2^r(x)$ stands for ``$x$ is a square and $r$ does not divide $x$''. The relation $\leq$ is $\{0,+,R_2\}$-uped in the class  $\Ucal$ of all structures $\Cfrak_r$ with $r\geq2$. 
\end{proposition}

Let us give an example where we do not have uniformity. Consider the language $\Lcal=\Lcal_A\cup\{\alpha\}$, where $\alpha$ is a symbol of constants. Consider the $\Lcal$-structures $\Mfrak^k=(\Z;0,1,+,\cdot,k)$, where $\alpha$ is interpreted as $k$ in each $\Mfrak^k$. The formula $x=k$ defines $k$ over $\Lcal_A$ in each $\Mfrak^k$, but there is no formula that uniformly $\Lcal_A$-defines $\alpha$ in the set $\{\Mfrak^k\colon k\in\Z\}$: such a formula $\varphi(x)$ would $\Lcal_A$-define $2$ in $\Mfrak^2_{\{\alpha\}}$ and $3$ in $\Mfrak^3_{\{\alpha\}}$, which is absurd as these two structures are the same (the ring of integers). Note that with this example, it is enough to consider two distinct structures. Next proposition shows that one can have uniformity in each finite subfamily of a family of structures but not in the whole family. The proof will be given in Section \ref{ProofCognac}.

\begin{proposition}\label{TimphuButhan}
Let $\Ccal$ be the set of all finite fields $\F_p$ of prime characteristic $p$. 
The relation ``to be a square'' is $\{0,1,+\}$-uped in any finite subfamily of $\Ccal$, but there is no infinite subfamily of $\Ccal$ where it is $\{0,1,+\}$-uped. Hence, in particular, multiplication is not $\{0,1,+\}$-uped in $\Ccal$.
\end{proposition}

A highly relevant result can be found in \cite{ChatzDriesMacintyre}, where it is shown that there is no formula in the language of rings that defines $\F_q$ in $\F_{q^2}$ for all but finitely many $q$ (here $q$ is any power of any prime).

We will now present one of the main tools that will allow us to obtain several uniform definitions. We first define a relation that has often been a key point to codify the integers in rings of functions of positive characteristic (see for example, by chronological order, \cite{Denef2}, \cite{PheidasPow}, \cite{PheidasInv}, \cite{KimRoush1}, \cite{Videla},\cite{Shlapentokh1}, \cite{PheidasZahidi1}, \cite{Shlapentokh2}, \cite{Eisentrager} and \cite{EisenShlapentokh}).

\begin{definition}\label{DefRA}
Let $R_A$ be the equivalence relation defined on $A$ by: 
$$
R_A(x,y)\textrm{ if and only if there exists } s\in\N \textrm{ such that either } 
y=x^{p^s} \textrm{ or } x=y^{p^s},
$$
where $p>0$ is the characteristic of $A$. For short we will say that ``there exists $s\in\Z$ such that $y=x^{p^s}$. 
\end{definition}

In order to show that the above relation is uped in several classes of structures, we need to introduce B\"uchi's problem. 

Let $A$ be a commutative ring with unit and of positive characteristic $p>2$. Let $C$ be a subring of $A$. If $M\geq3$, let us call an \emph{$M$-term B\"uchi sequence for $(A,C)$} a sequence of $M$ elements of $A$, not all in $C$, whose second difference of squares is the constant sequence $(2)$.\\  

\noindent \textbf{B\"uchi's Problem for Rings of Characteristic $p>2$:}\\
\noindent $\BP(A,C,M)$ \emph{Is it true that for all $N\geq M$, any $N$-term B\"uchi sequence $(x_n)$ of $(A,C)$ satisfies 
$$
x_n^2=(x+n)^{p^s+1}, \qquad n=1,\dots,N, 
$$
for some $x\in A$ and some non-negative integer $s$?}\\

\begin{notation}
If $\BP(A,C,M)$ has a positive answer for some $M$ then we will denote by $M_0(A,C)$ 
the least such $M$.
\end{notation}

Note that $M_0(A,C)$, if it exists, is always at most the characteristic $p$ of $A$ (as if there exists an $M$ greater than $p$ then the B\"uchi sequence is $p$-periodic; see \cite{PastenPheidasVidaux1}). 

We prove:

\begin{theorem}\label{TeoremonA}
If $\BP(A,C,M)$ has a positive answer then there exists a positive existential $\Lcal_A$-formula $\varphi_{M_0(A,C)}(x,y)$ with the following properties: 
\begin{enumerate}
\item\label{Aimp} If $R_A(x,y)$ holds then $A$ satisfies $\varphi_{M_0(A,C)}(x,y)$; and
\item\label{Aeq} if either $xy$ or $x+y$ is not in $C$ then: $R_A(x,y)$ holds if and only if $A$ satisfies $\varphi_{M_0(A,C)}(x,y)$. 
\end{enumerate}
\end{theorem}

In the cases relevant to this work, B\"uchi's problem is known to have a positive answer when $(A,C,M)$ is
\begin{enumerate}
\item $(F[z],F,14)$ for any field $F$ of characteristic $p\geq17$;
\item $(F(z),F,18)$ for any field $F$ of characteristic $p\geq19$;
\item $(K,F,312g+169)$ for any function field of a curve $K$ of genus $g$, with field of constants $F$, and of characteristic $p\geq 312g+169$.
\end{enumerate}

For a reference, see \cite{PheidasVidaux2} and \cite{PheidasVidaux2bis} for Items 1 and 2, and \cite{ShlapentokhVidaux} for Item 3.

In order to uniformly define the relation $R_A$ in some classes of structures, we need to introduce the following definition. 

\begin{definition}\label{pomo}
Let us call \emph{B\"uchi class} any class $\Ccal$ of pairs of rings such that there exists an integer $M$ so that $\BP(A,C,M)$ has a positive answer for any $(A,C)$ in the class. If $\Ccal$ is a B\"uchi class, we denote by $M(\Ccal)$ the maximum of the set 
$$
\{M_0(A,C)\colon (A,C)\textrm{ in the class }\Ccal\}
$$
and by $\bar\Ccal$ the class of structures $A$ such that $(A,C)$ is in $\Ccal$ for some $C$ (so $\bar\Ccal$ is the projection on the first component). 
\end{definition}

Note that $M(\Ccal)$ may be greater than some of the characteristics of the $A$ in $\bar\Ccal$ but this can happen for at most a finite number of characteristics. 

\begin{theorem}\label{TeoremonB}
Let $\Ccal$ be a B\"uchi class such that $C$ is a field for each pair $(A,C)$ in the class. Suppose that for each pair $(A,C)$ in the class $\Ccal$, $A$ is both an $\Lcal_T$-structure and an $\Lcal_z$-structure, where $T(x)$ is interpreted as ``$x$ is transcendental over $C$'' and $z$ is a symbol of constant interpreted by an element of $A$ transcendental over $C$. There exist a positive existential $\Lcal_T$-formula $\varphi_\Ccal^T(x,y)$ and a positive existential $\Lcal_z$-formula $\varphi_\Ccal^z(x,y)$ 
with the following properties:
\begin{enumerate}
\item $\varphi_\Ccal^T(x,y)$ uniformly defines $R_A$ in $\bar\Ccal$ (hence the collection of relations $R_A$ is $\Lcal_T$-uped in $\bar\Ccal$); and\label{BT}
\item $\varphi_\Ccal^z(x,y)$ uniformly defines $R_A$ in $\bar\Ccal$ (hence the collection of relations $R_A$ is $\Lcal_z$-uped in $\bar\Ccal$).\label{Bz}
\end{enumerate} 
\end{theorem}

Here are some known B\"uchi classes where Theorem \ref{TeoremonB} applies:

\begin{enumerate}
\item Any non-empty subclass of the class of pairs $(F[z],F)$ where $F[z]$ is a polynomial ring over a field $F$ of characteristic at least $17$. 
\item Any non-empty subclass of the class of pairs $(F(z),F)$ where $F(z)$ is a rational function field over a field $F$ of characteristic at least $19$. 
\item Given an integer $g_0\geq0$, any non-empty subclass of the class of pairs $(A,C)$ where $A$ is a function field of a curve of genus $g\leq g_0$ and of positive characteristic at least $312g+169$, with $C$ the field of constants of $A$.
\end{enumerate}

Theorem \ref{TeoremonB} is enough for our purposes but, for sake of completeness, we prove an analogous result for a relation weaker than $R_A$, but which can be applied in more general classes. 

\begin{theorem}\label{TeoremonC}
Let $\Ccal$ be a B\"uchi class. For each pair $(A,C)$ in the class $\Ccal$, suppose that $A$ is an $\Lcal_T$-structure where $T(x)$ is interpreted as ``$x\notin C$'' and $C$ has the following properties: 
\begin{itemize}
\item for all $x\in A$, if $2x\in C$ then $x\in C$; and 
\item for all $x\in A$, if $x^2\in C$ then $x\in C$.
\end{itemize}
Let $R_A^C$ be the relation defined by 
$$
R_A(x,y) \textrm{ holds, and either } x \textrm{ or } y \textrm{ is not in } C.
$$
There exists a positive existential $\Lcal_T$-formula $\psi_\Ccal^T(x,y)$ with the following property: $\psi_\Ccal^T(x,y)$ uniformly defines $R_A^C$ in $\bar\Ccal$ (hence the collection of relations $R_A^C$ is $\Lcal_T$-uped in $\bar\Ccal$). 
\end{theorem}

\begin{notation}\label{clason}
We will denote by $\Omega$ any class of $\Lcal_{z,\ord,\ne}$-structures such that there exists a B\"uchi class $\Ccal$ of pairs $(K,C)$, where $K$ is a function field of a curve of genus at most some fixed integer $g_0$ and $C$ is the constant field of $K$, such that for each $\Lcal_{z,\ord,\ne}$-structure $\Mfrak$ in $\Omega$, there exists a pair $(K,C)$ in $\Ccal$ such that:
\begin{itemize}
\item the base set $M$ of $\Mfrak$ is a subring of $K$ and contains $C$;
\item $M$ contains some local parameter $\xi$ at some prime divisor $\pfrak$;
\item $z$ is interpreted as $\xi$;
\item $\ord(x)$ is interpreted as ``the order of $x$ at $\pfrak$ is non-negative'';
\item $\ne$ is interpreted as usual.
\end{itemize}
\end{notation} 

Note that in the above notation, since there can be more than one choice of $\xi$ for a pair $(K,C)$, several $\Lcal_{z,\ord,\ne}$-structures $\Mfrak$ may correspond to the same pair $(K,C)$ in the B\"uchi class. Note also that Theorem \ref{TeoremonB} applies to the class of pairs $(A,C)$ where $A$ ranges in $\Omega$ and is seen as a ring. 

Next theorem gives uniform definitions in other types of classes of structures. 

\begin{theorem}\label{Teoremon1} 
Multiplication is uniformly positive existentially
\begin{enumerate}
\item $\Lcal^{*,+}$-definable in $\Ncal=\{(\N;0,1,+,\mid_p)\colon p\textrm{ is prime}\}$;\label{1N}
\item $\Lcal^*$-definable in $\Zcal=\{(\Z;0,1,+,\leq,\mid_p)\colon p\textrm{ is prime}\}$; and\label{1Z}
\item $\Lcal_T^*$-definable in $\Dcal=\{(\Z;0,1,+,\mid,\mid_p,\Z\smallsetminus\{-1,0,1\})\colon p\textrm{ is prime}\}$.\label{1D}
\end{enumerate}
\end{theorem}

Before stating our main results, we need to introduce the following definition. 

\begin{definition}
Consider two languages $\Lcal$ and $\Lcal'$. Let $\Mfrak$ be an $\Lcal$-structure and $\Ucal$ be a class of $\Lcal'$-structures. Let $\Gcal$ be a set of $\Lcal$-sentences and $\Gcal'$ a set of $\Lcal'$-sentences. We will say that $(\Gcal,\Mfrak)$ is \emph{uniformly encodable} in $(\Gcal',\Ucal)$ if there exists an algorithm $\Acal$ that, given a formula $F\in\Gcal$, returns a formula $\Acal(F)\in\Gcal'$ such that the following are equivalent:
\begin{enumerate}
\item $\Mfrak$ satisfies $F$.
\item Any structure $\Ufrak$ in $\Ucal$ satisfies $\Acal(F)$.
\item There exists a structure $\Ufrak$ in $\Ucal$ that satisfies $\Acal(F)$. 
\end{enumerate}
\end{definition}

\begin{remark}\label{RemUnionInter}
Note that a codification is uniform exactly when the formulas that are true in $\Mfrak$ are sent through the algorithm to  formulas in the intersection of the theories of the structures in $\Ucal$, while the formulas that are not true in $\Mfrak$ are sent to the complement of the union of these theories.
\end{remark}

\begin{remark}\label{RemFilter}
Let $\Acal$ be an algorithm that uniformly encodes a pair $(\Gcal,\Mfrak)$ in a pair $(\Gcal',\Ucal)$, where $\Gcal$ is a set of sentences over a language $\Lcal$ and $\Gcal'$ is a set of sentences over a language $\Lcal'$.
\begin{enumerate}
\item For any non-empty subset $\Gcal_0$ of $\Gcal$, the algorithm $\Acal$ uniformly encodes $(\Gcal_0,\Mfrak)$ in $(\Gcal',\Ucal)$.\label{fil1}
\item For any set $\Gcal'_0$ of $\Lcal'$-sentences that contains $\Gcal'$, the algorithm $\Acal$ uniformly encodes $(\Gcal,\Mfrak)$ in $(\Gcal'_0,\Ucal)$.\label{fil2}
\item For any non-empty subclass $\Ucal_0$ of $\Ucal$, the algorithm $\Acal$ uniformly encodes $(\Gcal,\Mfrak)$ in $(\Gcal',\Ucal_0)$.\label{fil3}
\item For any language $\Lcal''=\Lcal'\cup X$, with $\Lcal'\cap X=\varnothing$, the algorithm $\Acal$ uniformly encodes $(\Gcal,\Mfrak)$ in $(\Gcal',\Ucal^X)$, no matter the interpretation of the elements of $X$ given in the structures $\Ufrak^X$ of $\Ucal^X$.\label{fil4}
\item For any language $\Lcal''=\Lcal\smallsetminus X\ne\varnothing$, if the set of $\Lcal''$-formulas $\Gcal''$ which are in $\Gcal$ is non-empty, then the algorithm $\Acal$ uniformly encodes $(\Gcal'',\Mfrak_X)$ in $(\Gcal',\Ucal)$.\label{fil5}
\end{enumerate}
\end{remark}

Uniform encodability can be used to show very strong undecidability results in the following way (the proof will be given in Section \ref{UnifEncod}). 

\begin{theorem}\label{PPV}
Suppose that a pair $(\Gcal,\Mfrak)$ is uniformly encodable in a pair $(\Gcal',\Ucal)$ and that there is no algorithm to decide whether or not a formula $F$ in $\Gcal$ is true in $\Mfrak$. Let $\Ccal$ be a non-empty collection of non-empty subclasses of $\Ucal$. There is no algorithm to solve the following problem: 
\begin{quote}
\textbf{(P)} Given $F\in\Gcal'$, decide whether or not there exists a class $\Vcal$ in the collection $\Ccal$ such that every structure $\Ufrak$ in $\Vcal$ satisfies $F$. 
\end{quote}
\end{theorem}

\begin{theorem}\label{Teoremon2} 
The pair $(\fpe_{\Lcal_A},\N)$ is uniformly encodable in
\begin{enumerate}
\item $(\fpe_{\Lcal^*},\Zcal)$;\label{2Z}
\item $(\fpe_{\Lcal^{*,+}},\Ncal)$; and\label{2N}
\item $(\fpe_{\Lcal_{z,\ord,\ne}},\Omega)$ (where $\Omega$ is any class as defined in Notation \ref{clason}).\label{2O}
\item\label{2Ob} $(\fpe_{\Lcal},\Omega)$ with 
	\begin{enumerate}
	\item $\Lcal=\Lcal_{z,\ord}$ if $\ne$ is $\Lcal_{z,\ord}$-uped in $\Omega$.
	\item $\Lcal=\Lcal_{z,\ne}$ if $\ord$ is $\Lcal_{z,\ne}$-uped in $\Omega$.
	\item $\Lcal=\Lcal_{z}$ if $\ne$ and $\ord$ are $\Lcal_{z}$-uped in $\Omega$.
	\end{enumerate} 
\end{enumerate}
\end{theorem}

In the following corollary, we specify some classes $\Omega$ for which we do have uniform definition of $\ne$ or $\ord$ (for Item 1, we use Lemma \ref{polne}). 

\begin{corollary}\label{Teoremon2a} 
The pair $(\fpe_{\Lcal_A},\N)$ is uniformly encodable in the pairs
\begin{enumerate}
\item $(\Fcal_{\Lcal_z}^{\rm pe},\bar\Ccal)$, where $\Ccal$ is any B\"uchi class of pairs $(A,C)$ where $C$ is a field and $A$ is a polynomial ring over $C$ (in particular for the class of all polynomial rings of characteristic at least $17$). 
\item $(\fpe_{\Lcal_{z,\ord}},\bar\Ccal)$, where $\Ccal$ is any B\"uchi class of pairs $(A,C)$ where $C$ is a field and $A$ is a rational function field over $C$ (in particular for the class of all rational function fields of characteristic at least $19$). 
\item $(\fpe_{\Lcal_{z,\ord}},\bar\Ccal)$, where $\Ccal$ is any B\"uchi class of pairs $(A,C)$ where $C$ is a field and $A$ is a function field of a curve of genus at most some fixed integer $g_0$, with constant field $C$ (in particular for the class of all such function fields of genus at most $g_0$ whose characteristic is at least $312g+169$, where $g$ is the genus of the function field). 
\end{enumerate}
\end{corollary}

\begin{theorem}\label{Teoremon2b} 
The pair $(\fpe_{\Lcal_A},\Z)$ is uniformly encodable in
\begin{enumerate}
\item $(\fpe_{\Lcal^*},\Zcal)$;\label{2bZ}
\item $(\fpe_{\Lcal_T^*},\Dcal)$;\label{2bD}
\item $(\fpe_{\Lcal_T},\Ccal)$, where $\Ccal$ is the class of all polynomial rings over a field of odd positive characteristic, where $T(x)$ is interpreted in each structure as ``$x$ is non-constant''. \label{2bCT}
\end{enumerate}
\end{theorem}

Actually, the proof of Item \ref{2bCT} works (with only notational changes) in a similar way to prove that the pair $(\Fcal_{\Lcal_A},\Z)$ is uniformly encodable in $(\Fcal_{\Lcal_A},\Ccal)$, where $\Ccal$ is the class of all polynomial rings over a field of odd positive characteristic.

We obtain the following corollary from Theorem \ref{PPV} (choosing suitably the class $\Ccal$ of subclasses of $\Ucal$). 

\begin{corollary}\label{Teoremon3} 
Let $\Lcal$ and $\Ucal$ be such that the conclusion of Theorems \ref{Teoremon2} or \ref{Teoremon2b} hold and suppose that $\Ucal$ is infinite. There is no algorithm to decide whether or not a positive existential $\Lcal$-sentence is true for (for example):
\begin{enumerate}
\item some $\Ufrak$ in $\Ucal$ (this item does not require $\Ucal$ to be infinite),
\item all $\Ufrak$ in $\Ucal$,
\item infinitely many $\Ufrak$ in $\Ucal$,
\item all but possibly finitely many $\Ufrak$ in $\Ucal$,
\item each structure in a given subclass of $\Ucal$. 
\end{enumerate}
\end{corollary}

We finish this introduction with a short list of open problems.

\begin{enumerate}
\item Find a class of rational function fields in positive characteristic, with infinitely many distinct characteristics, where we can define uniformly the order over the language $\Lcal_z$, or show that there is no such class (make uniform the definition given in \cite{PheidasInv} or prove that it is not possible). 
\item Lower the bounds on the characteristic. 
\item Prove or disprove that there is a uniform definition of multiplication for the class of structures $\F_p$ over the language $\{0,1,+,P_2\}$, where $P_2(x)$ is interpreted in each $\F_p$ by ``$x$ is a square'' (it does not seem too difficult to prove the non-existence of a positive existential definition - see the proof of Proposition \ref{TimphuButhan} in Section \ref{ProofCognac}). 
\item Extend the result about $\ne$ in Lemma \ref{polne} to bigger classes of rings of algebraic functions. 
\end{enumerate}


\section{Examples of (non-)uniform definitions}\label{ProofCognac}

The proof of Lemma \ref{polne} is an easy adaptation of the proof of the analogous result over the integers (which we got from a talk by A. Shlapentokh). 

\begin{proof}[Proof of Lemma \ref{polne}]
Consider the following positive existential $\Lcal_z$-formula
$$
\varphi_{\ne}(t)\colon \exists x,u,v((zu-1)((z+1)v-1)=tx). 
$$
We prove that $\varphi_{\ne}(t)$ is satisfied in a polynomial ring $F[z]$, where $F$ is a field, if and only if $t$ is distinct from $0$. First note that it is clear that if the formula is satisfied then $t$ is not $0$ (since neither $z$ nor $z+1$ is invertible). 

Suppose that $t$ is non-zero. Since $F[z]$ is a unique factorization domain, we can write $t$ as $t_0t_1$ in such a way that $z$ does not divide $t_0$ and $z+1$ does not divide $t_1$. By B\'ezout's identity, there exist polynomials $u$, $x_u$, $v$ and $x_v$ such that $zu+t_0x_u=1$ and $(z+1)v+t_1x_v=1$. Therefore, we have
$$
(zu-1)((z+1)v-1)=t_0x_ut_1x_v=tx_ux_v,
$$
hence we can choose $x=x_ux_v$ for the formula to be satisfied. 
\end{proof}

\begin{proof}[Proof of Proposition \ref{TimphuButhan}]
Let $X$ be a non-empty finite set of prime numbers and let $q$ be its maximum. The quantifier free $\{0,1,+\}$-formula
$$
\bigvee_{i=0}^{q-1}x=i^2
$$
is satisfied in $\F_p$ if and only if $x$ is a square, for each $p$ in $X$. 

Suppose that $\varphi(x)$ is a positive existential $\{0,1,+\}$-formula that defines the relation ``$x$ is a square'' in $\F_p$ for all primes $p$ in an  infinite set $X$. The formula $\varphi(x)$ is logically equivalent to a formula of the form 
$$
\exists y_1\dots\exists y_n\bigvee_{i=1}^{r}L_i(x,y_1,\dots,y_n)
$$ 
where each $L_i$ is a formal system of linear equations. Hence, for each $p\in X$, the set of $x$ such that $\varphi(x)$ is true in $\F_p$, namely, the set of squares in $\F_p$, is the union of the projections on the variable $x$ of the zero locus $H_i^p$ of each $L_i$. Each $H_i^p$ is an affine linear subspace of $\F_p^{n+1}$ or the empty set. Since the projection $K_i^p$ of each $H_i^p$ on $x$ is an affine linear subspace of $\F_p$, it is either the whole of $\F_p$, a point or the empty set. From now on assume that $p$ is bigger than $2$. Since there are $\frac{p+1}{2}$ squares in $\F_p$, none of the $K_i^p$ can be the whole of $\F_p$ (as the union of the $K_i^p$ is the set of squares in $\F_p$). Hence the number $r$ of disjunctions in the formula $\varphi$ is at least $\frac{p+1}{2}$, which is absurd.
\end{proof}

The rest of this section is dedicated to the proof of Proposition \ref{Cognac}. If $A$ is a set of non-negative integers, then we define 
$$
A(n)=|A\cap\{1,2,\ldots,n\}|
$$ 
and
$$
\sigma(A)=\inf_{n>0}\frac{A(n)}{n}.
$$
The function $\sigma$ is known as the \emph{Shnirel'man density}. If $n\ge2$ and $A,A_1,\dots,A_n$ are sets of positive integers, we will write
$$
\sum^n_{i=1}A_i=\left\{\sum^n_{i=1}\alpha_i\colon \alpha_i\in A_i\right\}
$$
and $nA$ is the sum of $n$ copies of $A$. 

The two following fundamental results on Shnirel'man density can be found in  \cite[Chapter 11, Section 3]{Nathanson}.

\begin{lemma}\cite[Lemma 11.2]{Nathanson}\label{dens2} 
If $A$ and $B$ are sets of non-negative integers such that $0\in A\cap B$, $\sigma(A)>\frac{1}{2}$ and $\sigma(B)>\frac{1}{2}$ then $A+B$ is the set of non-negative integers.
\end{lemma}

\begin{theorem}\cite[Theorem 11.2]{Nathanson}\label{dens1} If $A_1,\ldots,A_t$ are sets of non-negative integers containing $0$, then we have 
$$
1-\sigma\left(\sum^t_{i=1}A_i\right)\le\prod_{i=1}^t(1-\sigma(A_i)).
$$
\end{theorem}

By a theorem of Linnik, given any set $B$ of non-negative integers with $\sigma(B)>0$,  the set $A=\{x^2\colon x\in B\}$ is a basis of finite order, that is, each positive integer is a (uniformly) bounded sum of elements in $A$. We want to show that the bound is the same for certain family of sets $B$.

\begin{theorem}\label{apendix}
Let $u\ge 2$ be an integer, and 
$$
C(u)=\{n\in\Z\colon u\nmid n\}.
$$ 
Each non-negative integer is the sum of at most $5940$ squares of elements in $C(u)$.
\end{theorem}

Using this result (proven below), Proposition \ref{Cognac} follows easily:

\begin{proof}[Proof of Proposition \ref{Cognac}] By Theorem \ref{apendix}, the following positive existential $\{0,+,R_2\}$-formula uniformly defines the relation $x\ge 0$ in the class of structures $\Cfrak_r$ (which is enough to prove the result):
$$
\phi(x): \exists x_1,\ldots,x_{5940} \bigwedge_{i=1}^{5940}(R_2(x_i)\vee x_i=0)\wedge x=\sum_{i=1}^{5940}x_i
$$
where we recall that $R_2(x)$ is interpreted as ``$x$ is a square and $r$ does not divide $x$'' in $\Cfrak_r$. 
\end{proof}

We need two lemmas before we can prove Theorem \ref{apendix}.
\begin{lemma}\label{cubo}
Let $d,k$ be positive integers, let 
$$
A_i=\{z_i+1,\ldots,z_i+k\}
$$ 
for $1\le i\le d$ be sets of $k$ consecutive integers and let 
$$
B_i=\{z_i+1,\ldots,z_i+k-1\}
$$ 
(take $B_i$ empty if $k=1$). Suppose that we have a set $U\subseteq \R^d$ satisfying the following:
\begin{itemize}
\item[(1)] $U$ is non-empty,
\item[(2)] $U$ is convex,
\item[(3)] $\pi_i(U)= \pi_i(H_{z_i+1}^i\cap U)$ for $1\le i\le d$, where $\pi_i:\R^d\to\R^{d-1}$ deletes the $i$-th coordinate and $H_x^{i}\subseteq \R^d$ is the hyperplane $x_i=x$.
\end{itemize}
Then we have
$$
(k-1)^d|U\cap\prod_{i=1}^d A_i|\le k^d|U\cap\prod_{i=1}^d B_i|.
$$ 
\end{lemma}
\begin{proof}
We fix $k\ge 1$. Up to translation by the vector $(z_1,\dots,z_d)$ we can assume $z_i=0$ for each $i$, hence 
$$
A_i=A=\{1,\ldots,k\}
$$ 
and 
$$
B_i=B=\{1,\ldots,k-1\}
$$ 
for each $i$. Define 
$$
1_d=(1,\ldots,1)
$$ 
and observe that $1_d$ belongs to $U$ (otherwise $U$ would be empty by the hypothesis (3)). Let $a_d\ge 1$ be a real number such that 
$$
(1,\ldots,1,a_d)\in U
$$ 
(this is possible because $1_d\in U$) and such that, if 
$$
(1,\ldots,1,l)\in U
$$ 
for $l\in A$ then $a_d\ge l$ (this can be done because $U$ is convex).

The proof goes by induction on $d$. Observe that for $d=1$ the set $U$ is just an interval containing $1$, thus the desired inequality clearly holds.
Assume that the result is true for $d=n-1\ge 1$ and consider a set $U\subseteq \R^n$ satisfying the hypotheses. For the rest of the proof, the set $H_x^n$ will be considered inside $\R^n$. It is easy to see that, for any $x\in[1,a_n]$ the set 
$$
U_x=\pi_n(U\cap H_x^n)\subseteq\R^{n-1}
$$ 
satisfies the hypotheses of the Lemma with $d=n-1$ and $z=0$, hence
$$
(k-1)^{n-1}|A^{n-1}\cap U_x|\le k^{n-1}|B^{n-1}\cap U_x|.
$$
For $x\in A$ we have
$$
\pi_n(H_x^n\cap A^n\cap U)=A^{n-1}\cap U_x
$$ 
hence
$$
|H_x^n\cap A^n\cap U|=|A^{n-1}\cap U_x|
$$ 
and similarly for $x\in B$ we have
$$
|H_x^n\cap B^n\cap U|=|B^{n-1}\cap U_x|.
$$ 
In particular, for $x\in B$ (hence also $x\in A$), we have
\begin{eqnarray}\label{cubos1}
(k-1)^{n-1}|H_x^n\cap A^n\cap U|\le k^{n-1}|H_x^n\cap B^n\cap U|.
\end{eqnarray}
The hypothesis (2) and (3) on $U$ implies 
$$
\pi_n\left(H_k^n\cap A^{n}\cap U\right)\times A\subseteq  A^n \cap U
$$ 
which gives us
\begin{eqnarray}\label{cubos2}
|H_k^n\cap A^{n}\cap U|\le \frac{1}{k}|A^n \cap U|.
\end{eqnarray}
Using the Inequalities \eqref{cubos1} and \eqref{cubos2} we obtain:
$$
\begin{aligned}
(k-1)^{n}|A^{n}\cap U|&=(k-1)\sum_{x\in A}(k-1)^{n-1}|H_x^n\cap A^n \cap U|\\
&=(k-1)^{n}|H_k^n\cap A^{n}\cap U|+\\
&\hspace{16pt}(k-1)\sum_{x\in B}(k-1)^{n-1}|H_x^n\cap A^n \cap U|\\
&\le(k-1)^{n}\frac{1}{k}|A^n \cap U|+ (k-1)\sum_{x\in B}k^{n-1}|H_x^n\cap B^{n}\cap U|\\
&=(k-1)^{n}\frac{1}{k}|A^n \cap U|+(k-1)k^{n-1}|B^n\cap U|
\end{aligned}
$$
hence
$$
(k(k-1)^n-(k-1)^n)|A^{n}\cap U|\le (k-1)k^n |B^n\cap U|
$$
and we obtain finally
$$
(k-1)^{n}|A^{n}\cap U|\le k^n|B^n\cap U|.
$$
\end{proof}

Let $d$ and $k$ be positive integers, and $r$ a positive real number. We let
$$
\begin{aligned}
L_{d}(r)&=\{v=(v_1,\ldots,v_d)\in\Z^d\colon \|v\|_2\le r,v_i>0\mbox{ for }1\le i\le d\}\\
L_{d,k}(r)&=\{v=(v_1,\ldots,v_d)\in\Z^d\colon \|v\|_2\le r,v_i>0, k\nmid v_i\mbox{ for }1\le i\le d\}.
\end{aligned}
$$
\begin{lemma}\label{cubitos}
We have 
$$
k^d|L_{d,k}(r)|\ge (k-1)^d|L_{d}(r)|.
$$
\end{lemma}
\begin{proof}
Let $U=D(0,r)$ be the $d$-dimensional closed euclidean ball of radius $r$. Take integers $z_i\ge 0$ congruent to $0$ modulo $k$ for $1\le i\le d$ such that 
$$
(z_1+1,\ldots,z_d+1)\in U,
$$ 
(if this is not possible - for $r$ being too small - then the conclusion follows). Write $z=(z_1,\ldots,z_d)$ and define the sets $A_i=A_i(z)$ and $B_i=B_i(z)$ as in Lemma \ref{cubo}. It is clear that, as $z$ ranges over all the possible choices then the sets 
$$
U\cap \prod_{i=1}^d A_i(z)
$$ 
form a partition of $L_{d}(r)$ and the sets 
$$
U\cap \prod_{i=1}^d B_i(z)
$$ 
form a partition of $L_{d,k}(r)$. For each fixed $z$, let 
$$
U_z=U\cap \{(x_1,\dots,x_d)\colon x_i\ge z_i \textrm{ for each } i\}.
$$
Note that 
$$
\left|U\cap \prod_{i=1}^d A_i\right|=\left|U_z\cap \prod_{i=1}^d A_i\right|\quad\textrm{and}\quad 
\left|U\cap \prod_{i=1}^d B_i\right|=\left|U_z\cap \prod_{i=1}^d B_i\right|
$$
and note also that the hypothesis in Lemma \ref{cubo} are satisfied for $U_z$, $A_i(z)$ and $B_i(z)$. The result follows.
\end{proof}

\begin{proof}[Proof of Theorem \ref{apendix}]
Let $r(n)$ be the number of ordered $6$-tuples of integers $(x_1,\ldots,x_6)$ such that 
$$
n=\sum x_i^2,
$$ 
and let $r'(n)$ be the number of ordered $6$-tuples of integers having their non-zero coordinates in $C(u)$ and satisfying the same condition. Write
$$
R(n)=\sum_{k=0}^n r(k)\quad\textrm{and}\quad R'(n)=\sum_{k=0}^n r'(k).
$$ 

Observe that $R(n)$ is the number of integer points in the $6$-dimensional closed euclidean ball $B(0,\sqrt{n})$ of radius $\sqrt{n}$. 

If $z\in\R^6$ define the \textit{box centered at} $z$ as the closed ball of radius $1/2$ in the $\infty$-norm centered at $z$ and write it $B_z$. Observe that, if $V$ is a set of $N$ integer points in $\R^6$, then
$$
N=\vol\left(\bigcup_{z\in V}B_z\right).
$$
Given $n>0$ define 
$$
I(n)=\Z^6\cap B(0,\sqrt{n})
$$ 
and 
$$
I'(n)=\{v\in I_n\colon u \mbox{ does not divide the nonzero coordinates of } v\}
$$
hence we have
$$
R(n)=|I(n)|=\vol\left(\bigcup_{z\in I(n)}B_z\right)
$$
and
$$
R'(n)=|I'(n)|=\vol\left(\bigcup_{z\in I'(n)}B_z\right)
$$
Moreover, decomposing $I(n)$ and $I'(n)$ in lower dimensional parts we have
$$
R(n)=1+\sum_{d=1}^{6}2^d\binom{6}{d}|L_{d}(\sqrt{n})|
$$
and
$$
R'(n)=1+\sum_{d=1}^{6}2^d\binom{6}{d}|L_{d,u}(\sqrt{n})|, 
$$
where $\binom{6}{d}$ counts the number of non-zero components and $2^d$ the distribution of the signs. Hence, by Lemma \ref{cubitos} we get
$$
\begin{aligned}
R'(n)&=1+\sum_{d=1}^{6}2^d\binom{6}{d}|L_{d,u}(\sqrt{n})|\\
&\ge 1+\sum_{d=1}^{6}\left(\frac{u-1}{u}\right)^d 2^d\binom{6}{d}|L_{d}(\sqrt{n})|\\
&\ge \left(\frac{u-1}{u}\right)^6 R(n).
\end{aligned}
$$

We have
\begin{equation}\label{madre}
B\left(0,\sqrt{n}-\frac{\sqrt{6}}{2}\right)\subseteq \bigcup_{z\in I(n)}B_z
\end{equation}
since if 
$$
||v||_2\le \sqrt{n}-\frac{\sqrt{6}}{2}
$$ 
then the nearest lattice point to $v$ is at a distance at most $\sqrt{6}/2$. Therefore, we have
$$
R(n)\ge\vol\left(B\left(0,\sqrt{n}-\frac{\sqrt{6}}{2}\right)\right)=\frac{\pi^3}{6}\left(\sqrt{n}-\frac{\sqrt{6}}{2}\right)^6
$$
which gives a lower bound that will allow us to conclude.
$$
R'(n)\ge\left(\frac{u-1}{u}\right)^6\frac{\pi^3}{6}\left(\sqrt{n}-\frac{\sqrt{6}}{2}\right)^6.
$$

Let us now look for an upper bound. Given $n\ge 1$, let $m(n)$ be the number of integers $k$ in $\{0,1,\ldots,n\}$ satisfying $r'(k)=0$ and write $X(n)$ for the set of these integers. Note that 
\begin{itemize}
\item $r'(0)=1\ne 0$,
\item for $n>0$ we have $r(n)<40n^2$ (see \cite[Theorem 14.6]{Nathanson}) and 
\item $r'(n)\le r(n)$. 
\end{itemize}
Therefore, we have the following upper bound for $n\ge 1$:
$$
\begin{aligned}
R'(n)&\le 1+ \sum_{\substack{1\le k\le n\\ k\notin X(n)}}r(k)\\
&< 1+40\sum_{\substack{1\le k\le n\\ k\notin X(n)}}k^2\\
&\le 1+40\sum_{k=m(n)+1}^n k^2\\
&=1+40\left(\frac{2n^3+3n^2+n}{6}-\frac{2m(n)^3+3m(n)^2+m(n)}{6}\right)
\end{aligned}
$$

Set 
$$
S=\{x^2\colon x\in C(u)\}\cup\{0\}
$$ 
and $A=6S$. Since $m(n)=n-A(n)$, we use for $n\ge 1$ the upper and lower bounds obtained for $R'(n)$ to get
$$
\begin{aligned}
40&\left(\frac{2n^3+3n^2+n}{6}-\frac{2(n-A(n))^3+3(n-A(n))^2+n-A(n)}{6}\right) +1>\\ 
&\hspace{80pt}\left(\frac{u-1}{u}\right)^6\frac{\pi^3}{6}(\sqrt{n}-\sqrt{6}/2)^6. 
\end{aligned}
$$
Working out the left hand side one obtains
$$
\begin{aligned}
&\frac{40}{3}A(n)^3- 20(2n+1) A(n)^2+\frac{20(6n^2+6n+1)}{3}A(n)+1>\\
 &\hspace{80pt}\left(\frac{u-1}{u}\right)^6\frac{\pi^3}{6}(\sqrt{n}-\sqrt{6}/2)^6. 
\end{aligned}
$$
Let $\sigma_n$ be such that $A(n)=\sigma_nn$. Note that $0<\sigma_n\le 1$ (recall that $n\ge 1$ and $1\in A$). Since $u\ge 2$ we have 
$$
\left(\frac{u-1}{u}\right)^6\frac{\pi^3}{6}>0.08,
$$ 
hence for $n>3$ 
\begin{equation}
\begin{aligned}\label{desigualdadcubos}
&\frac{40}{3}\sigma_n(\sigma_n^2-3\sigma_n+3)n^3+20\sigma_n(2-\sigma_n)n^2+\frac{20\sigma_n}{3}n+1>\\
&\hspace{80pt}0.08\left(\sqrt{n}-\frac{\sqrt{6}}{2}\right)^6. 
\end{aligned}
\end{equation}

If for some $n\ge 500$ we have $\sigma_n\le 0.0014$ then the above inequality and elementary calculus gives a contradiction. Hence $\sigma_n>0.0014$ for each $n\ge 500$. On the other hand, as $1\in A$ we have 
$$
\sigma_n\ge \frac{1}{499}>0.0014
$$ 
for $n=1,2,\ldots,499$. Note that these bounds are far from being optimal, but they are enough for our purposes.

This proves that $\sigma_n> 0.0014$ for each $n\ge 1$. Therefore we have $\sigma(A)\ge 0.0014$ and Theorem \ref{dens1} implies 
$$
\sigma(495 A)\ge 1-(1-0.0014)^{495}>0.5.
$$ 
By Lemma \ref{dens2}, 
$$
5940S=990A=2(495A)
$$ 
is the set of non-negative integers.
\end{proof}

\section{Uniform encodings}\label{UnifEncod}

\subsection{Proof of Theorem \ref{PPV} and Corollary \ref{Teoremon3}}\label{PPVMachine}

\begin{proof}[Proof of Theorem \ref{PPV}]
Suppose that under the hypothesis of the theorem there exists an algorithm $\Acal$ to solve Problem (P), and let $\Bcal$ be the algorithm that uniformly encodes $(\Gcal,\Mfrak)$ in $(\Gcal',\Ucal)$. Let us show that the algorithm obtained by first applying $\Bcal$ and then $\Acal$ decides whether or not a formula in $\Gcal$ is satisfied by $\Mfrak$ (which is absurd). Let $F$ be a formula in $\Gcal$ and apply $\Acal$ to the output $G$ of $F$ after applying $\Bcal$. 
\begin{itemize}
\item if the answer is YES then there exists a non-empty class $\Vcal$ in the collection $\Ccal$ such that every structure $\Ufrak$ in $\Vcal$ satisfies $G$. In particular, there exists at least one structure in $\Ucal$ satisfying $G$. Therefore, $\Mfrak$ satisfies $F$ (by definition of uniform encodability).
\item if the answer is NO then for each class $\Vcal$ in the non-empty collection $\Ccal$, there exists at least one structure $\Ufrak$ in $\Vcal$ not satisfying $G$. In particular, there exists at least one structure in $\Ucal$ not satisfying $G$. Therefore, $\Mfrak$ does not satisfy $F$ (by definition of uniform encodability).
\end{itemize}
\end{proof}

\begin{proof}[Proof of Corollary \ref{Teoremon3}]
We list by item the collection $\Ccal$ needed to apply Theorem \ref{PPV}. The collection $\Ccal$ consists respectively of:
\begin{enumerate}
\item all classes containing exactly one structure in $\Ucal$;
\item the class $\Ucal$
\item all infinite subclasses of $\Ucal$,
\item all cofinite subclasses of $\Ucal$,
\item the given subclass of $\Ucal$. 
\end{enumerate}
\end{proof}

\subsection{Techniques for uniform encodings}

Following Cori and Lascar \cite{CoriLascar} we recall the following notation and definitions. 

\begin{notation}
\begin{enumerate}
\item If $\Ufrak$ is an $\Lcal$-structure, then for each symbol $\alpha$ of $\Lcal$, we will write $\alpha^\Ufrak$ for the interpretation of $\alpha$ in $\Ufrak$. 
\item Let $f\colon\Ufrak\rightarrow\Wfrak$ be a morphism of $\Lcal$-structures. We will say that $f$ is an \emph{$\Lcal$-monomorphism} if for each relation symbol $R$ we have: for all $x_1,\dots,x_n\in\Ufrak$
$$
R^{\Ufrak}(x_1,\dots,x_n) \textrm{ holds if and only if } R^{\Wfrak}(f(x_1),\dots,f(x_n)) \textrm{ holds}.
$$
\item An $\Lcal$-isomorphism is an $\Lcal$-monomorphism which is onto. 
\end{enumerate}
\end{notation}

Note that sometimes $\Lcal$-monomorphisms are called $\Lcal$-embeddings.

\begin{definition}\label{obvio}
Let $\Ufrak$ be an $\Lcal$-structure and $\Lcal'$ be a language. Suppose that there exists a bijection $f\colon\Lcal\rightarrow\Lcal'$ which sends symbols of constants to symbols of constants, and for each natural number $n\ge1$, symbols of $n$-ary relations to symbols of $n$-ary relations, and  symbols of $n$-ary functions to symbols of $n$-ary functions. Let $\Ufrak'$ be the $\Lcal'$-structure with same base set as $\Ufrak$ and where each symbol $f(\alpha)$ from $\Lcal'$ is interpreted by $\alpha^\Ufrak$. Given $\Ufrak$ and $f\colon\Lcal\rightarrow\Lcal'$ as above, we will refer to $\Ufrak'$ as to the $(\Ufrak,f)$-induced $\Lcal'$-structure. Moreover, in this context, we will denote by $\Acal_{\Lcal}^{\Lcal'}$ the algorithm that transforms a formula over $\Lcal$ into a formula over $\Lcal'$ (simply using the bijection $f$). Note that for every formula $F$ over $\Lcal$, we have: $\Ufrak$ satisfies $F$ if and only if $\Ufrak'$ satisfies $\Acal_{\Lcal}^{\Lcal'}(F)$. 
\end{definition}

\begin{proposition}\label{Atransform}
Let $\alpha\in\Lcal$ be uniformly $\Lcal\smallsetminus X$-definable in a class $\Ucal$ of $\Lcal$-structures. There exists an algorithm $\Acal_X^\alpha$ that, given an $\Lcal$-sentence $F$, returns an $\Lcal\smallsetminus X$-sentence $\Acal_X^\alpha(F)$ such that $\Ufrak$ satisfies $F$ if and only if $\Ufrak_X$ satisfies $\Acal_X^\alpha(F)$ for all structures $\Ufrak\in\Ucal$. Moreover, if $\alpha$ is (respectively, positive) existentially definable and $F$ is (respectively, positive) existential then $\Acal_X^\alpha(F)$ is (respectively, positive) existential. 
\end{proposition}
\begin{proof}
In each $\Lcal$-sentence $F$, replace $\alpha$ by the formula that defines it uniformly.
\end{proof}

\begin{notation}
\begin{enumerate}
\item If $\Acal$ and $\Bcal$ are two algorithms such that the set of outputs of $\Bcal$ is included in the set of inputs of $\Acal$, we will denote by $\Acal\circ\Bcal$ the algorithm that first applies $\Bcal$ and then $\Acal$. 
\item Let $\alpha_1,\dots,\alpha_n\subset\Lcal$ be uniformly $\Lcal\smallsetminus X$-definable in a class $\Ucal$ of $\Lcal$-structures. We will denote by 
$$
\Acal_X^{\alpha_1,\dots,\alpha_n}
$$ 
the algorithm $\Acal_X^{\alpha_1}\circ\dots\circ\Acal_X^{\alpha_n}$.
\end{enumerate}
\end{notation}

\begin{proposition}\label{Quentin}
Let $\Gcal$, $\Gcal'$ and $\Gcal''$ be sets of sentences over $\Lcal$, $\Lcal'$ and $\Lcal''$ respectively. Let $\Mfrak$ be an $\Lcal$-structure, $\Ucal$ a class of $\Lcal'$-structures and $\Vcal$ a class of $\Lcal''$-structures. Let $(\Vcal_\Ufrak)$ be a partition of $\Vcal$ indexed by a subclass $\Ucal_{\rm ind}$ of $\Ucal$. If
\begin{itemize}
\item $(\Gcal,\Mfrak)$ is uniformly encodable in $(\Gcal',\Ucal)$ by an algorithm $\Acal$; and
\item there exists an algorithm $\Bcal$ such that for each $\Ufrak$ in $\Ucal_{\rm ind}$, the pair $(\Gcal',\Ufrak)$ is uniformly encodable in $(\Gcal'',\Vcal_\Ufrak)$ by $\Bcal$
\end{itemize}
then $(\Gcal,\Mfrak)$ is uniformly encodable in $(\Gcal'',\Vcal)$ by the algorithm $\Bcal\circ\Acal$. 
\end{proposition}
\begin{proof}
We may visualize the statement schematically as
$$
(\Gcal,\Mfrak)\xrightarrow{\Acal}(\Gcal',\Ucal)\supseteq(\Gcal',\Ucal_{\rm ind})
\xrightarrow{\Bcal}\left(\Gcal'',\bigcup_{\Ufrak\in\Ucal}\Vcal_\Ufrak\right)=\left(\Gcal'',\Vcal\right)
$$
and observe that by Item \ref{fil3} of Remark \ref{RemFilter}, $\Acal$ uniformly encodes $(\Gcal,\Mfrak)$ in $(\Gcal',\Ucal_{\rm ind})$. Let $F$ be an $\Lcal$-sentence. 

Let us prove that if $\Mfrak$ satisfies $F$ then each $\Vfrak$ in $\Vcal$ satisfies $\Bcal(\Acal(F))$. Since $\Vfrak$ is in $\Vcal$, it is in some $\Vcal_{\Ufrak}$, for some $\Ufrak$ in $\Ucal_{\rm ind}$. Since $\Mfrak$ satisfies $F$ and $(\Gcal,\Mfrak)$ is uniformly encodable in $(\Gcal',\Ucal)$ by $\Acal$, $\Ufrak$ satisfies $\Acal(F)$, and since $(\Gcal',\Ufrak)$ is uniformly encodable in $(\Gcal'',\Vcal_\Ufrak)$ by $\Bcal$, $\Vfrak$ satisfies $\Bcal(\Acal(F))$. 

Let us prove that if $\Vfrak$ satisfies $\Bcal(\Acal(F))$ for some $\Vfrak$ in $\Vcal$ then $\Mfrak$ satisfies $F$. Let $\Ufrak$ in $\Ucal$ be such that $\Vfrak$ is in $\Vcal_\Ufrak$. Since $\Vfrak$ satisfies $\Bcal(\Acal(F))$, also $\Ufrak$ satisfies $\Acal(F)$, hence $\Mfrak$ satisfies $F$. 
\end{proof}

We see from the proof that the above proposition actually requires only a covering of $\Vcal$ instead of a partition.

We now describe the general strategy that we will use several times in order to uniformly encode the natural numbers in classes of structures. Depending on the class in which we want to encode we will sometimes need two steps. \\

\noindent\textbf{One step encoding process.} Let $\Mfrak$ be a $\bar\Lcal$-structure.
In order to prove that a pair $(\fpe_{\bar\Lcal},\Mfrak)$ is uniformly encodable in a pair $(\fpe_{\Lcal},\Ucal)$ we will enlarge the language $\Lcal$ by a set of symbols $X=\{\alpha_1,\dots,\alpha_n\}$ and consider an interpretation of each element of $X$ in each $\Ufrak\in\Ucal$ so that
\begin{enumerate}
\item it is easy to prove that $(\fpe_{\bar\Lcal},\Mfrak)$ is uniformly encodable in $(\fpe_{\Lcal\cup X},\Ucal^X)$, say by an algorithm $\Acal$; and \label{onestepa}
\item each $\alpha$ in $X$ is uniformly positive existentially $\Lcal$-definable in $\Ucal$.\label{onestepb}
\end{enumerate}
From Item \ref{onestepb} we can apply Proposition \ref{Atransform}, and we will then be able to conclude by using Item \ref{fil2} of Remark \ref{RemFilter} since 
$$
\Acal_X^{\alpha_1,\dots,\alpha_n}(\fpe_{\Lcal\cup X})
$$ 
is included in $\fpe_\Lcal$. Schematically, we perform (with some obvious abuses of notation):
$$
(\fpe_{\bar\Lcal},\Mfrak)\xrightarrow{\,\,\,\Acal\,\,\,}(\fpe_{\Lcal\cup X},\Ucal^X)
\xrightarrow{\Acal_X^{\alpha_1,\dots,\alpha_n}}(\Acal_X^{\alpha_1,\dots,\alpha_n}(\fpe_{\Lcal}),\Ucal)
\subseteq(\fpe_{\Lcal},\Ucal)
$$
and we deduce that the algorithm $\Acal_0=\Acal_X^{\alpha_1,\dots,\alpha_n}\circ\Acal$ uniformly encodes $(\fpe_{\bar\Lcal},\Mfrak)$ in $(\fpe_{\Lcal},\Ucal)$.\\

\noindent\textbf{Two steps encoding process.} Let $\Mfrak$ be a $\bar\Lcal$-structure. Suppose that we have an algorithm $\Acal_0$ given by the ``one step encoding process'' to uniformly encode $(\fpe_{\bar\Lcal},\Mfrak)$ in a pair $(\fpe_{\Lcal},\Ucal)$ and that we want to encode it in another pair $(\fpe_{\Lcal'},\Vcal)$, for some class $\Vcal$ of $\Lcal'$-structures. Assume that we can find a partition $(\Vcal_\Ufrak)$ of $\Vcal$ indexed by a subclass $\Ucal_{\rm ind}$ of $\Ucal$ (note that by Item \ref{fil1} of Remark \ref{RemFilter}, $\Acal_0$ uniformly encodes $(\fpe_{\bar\Lcal},\Mfrak)$ in $(\fpe_{\Lcal},\Ucal_{\rm ind})$). In order to apply Proposition \ref{Quentin}, we need to find an algorithm $\Bcal$ such that for each $\Ufrak\in\Ucal_{\rm ind}$, $(\fpe_{\Lcal},\Ufrak)$ is uniformly encodable in $(\fpe_{\Lcal'},\Vcal_\Ufrak)$ by $\Bcal$. We then need to enlarge the language $\Lcal'$ by a set of symbols $Y=\{\beta_1,\dots,\beta_n\}$ and consider an interpretation of each element of $Y$ in each $\Vfrak\in\Vcal$ so that we can easily find an algorithm $\Bcal'$ such that
\begin{enumerate}
\item for each $\Ufrak\in\Ucal_{\rm ind}$, $(\fpe_{\Lcal},\Ufrak)$ is uniformly encodable in $(\fpe_{\Lcal'\cup Y},\Vcal_\Ufrak)$ by $\Bcal'$; and
\item each $\beta$ in $Y$ is uniformly positive existentially $\Lcal'$-definable in $\Vcal$.
\end{enumerate}
At this point, the algorithm $\Bcal$ is the composition 
$$
\Acal_Y^{\beta_1,\dots,\beta_n}\circ\Bcal'.
$$ 
We will then be able to conclude using Item \ref{fil2} of Remark \ref{RemFilter} since 
$$
\Acal_Y^{\beta_1,\dots,\beta_n}(\fpe_{\Lcal'\cup Y})
$$ 
is included in $\fpe_{\Lcal'}$. So the composition $\Bcal\circ\Acal_0$ uniformly encodes $(\fpe_{\bar\Lcal},\Mfrak)$ in $(\fpe_{\Lcal'},\Vcal)$. Schematically we obtain:
$$
(\fpe_{\bar\Lcal},\Mfrak)\xrightarrow{\,\,\,\Acal_0\,\,\,}(\fpe_{\Lcal},\Ucal)\supseteq (\fpe_{\Lcal},\Ucal_{\rm ind})
$$
$$
\xrightarrow{\,\,\,\Bcal'\,\,\,}
\left(\fpe_{\Lcal'\cup Y},\bigcup_{\Ucal_{\rm ind}}\Vcal_{\Ufrak}^{Y}\right)
\xrightarrow{\Acal_Y^{\beta_1,\dots,\beta_n}}(\Acal_Y^{\beta_1,\dots,\beta_n}(\fpe_{\Lcal'\cup Y}),\Vcal)
\subseteq(\fpe_{\Lcal'},\Vcal)
$$
and we deduce that the algorithm 
$$
\Acal_Y^{\beta_1,\dots,\beta_n}\circ\Bcal'\circ\Acal_0
$$ 
uniformly encodes $(\fpe_{\bar\Lcal},\Mfrak)$ in $(\fpe_{\Lcal'},\Vcal)$.\\

In order to find the algorithm $\Bcal'$ in the above process, we will need the following lemmas. They are certainly well known, but we decided to include them as we could not find a reference with the precise statements we needed. Let us introduce first some notation. 

\begin{notation}
Given a map $f\colon X\rightarrow Y$, we will denote by 
\begin{itemize}
\item $\sim_f$ the equivalence relation on $X$ defined by: $a\sim_f b$ if and only if $f(a)=f(b)$; 
\item $X_f$ the quotient set $\frac{X}{\sim_f}$;
\item $\pi_f$ the canonical projection 
$$
\pi_f\colon X\rightarrow X_f;
$$ 
\item $\bar f$ the unique map
$$
\bar f\colon X_f\rightarrow Y
$$ 
such that $\bar f\circ\pi_f=f$; and 
\item if $R$ is an $n$-ary relation on $X$ then $R_f$ will denote the $n$-ary relation on $X_f$ defined by: $R_f(\pi_f(x_1),\dots,\pi(x_n))$ if and only if there exist $u_1\in\pi_f(x_1)$, \dots, $u_n\in\pi_f(x_n)$ such that $R(u_1,\dots,u_n)$.
\end{itemize}
\end{notation}

\begin{lemma}\label{SetMon}
Let $X$ and $Y$ be sets together with $n$-ary relations $R$ on $X$ and $S$ on $Y$. Let $f\colon X\rightarrow Y$ be a function. If the function $f$ satisfies: 
\begin{enumerate}
\item if $R(x_1,\dots,x_n)$ holds then $S(f(x_1),\dots,f(x_n))$ and
\item if $S(f(x_1),\dots,f(x_n))$ holds then $R_f(\pi_f(x_1),\dots,\pi_f(x_n))$, 
\end{enumerate}
then the relation $R_f$ satisfies: 
$$
R_f(\bar x_1,\dots,\bar x_n) \textrm{ holds if and only if } S(\bar f(\bar x_1),\dots,\bar f(\bar x_n)) \textrm{ holds}
$$
for all $\bar x_1,\dots,\bar x_n\in X_f$.
\end{lemma}
\begin{proof}
We need only to prove the implication from left to right. Let 
$$
\bar x_1,\dots,\bar x_n\in X_f
$$ 
and suppose that 
$$
R_f(\bar x_1,\dots,\bar x_n)
$$ 
holds. By definition of $R_f$, there exist 
$$
u_1\in\bar x_1,\, \dots,\,u_n\in\bar x_n
$$ 
such that 
$$
R(u_1,\dots,u_n)
$$ 
holds. By Condition 1, 
$$
S(f(u_1),\dots,f(u_n))
$$ 
holds, and since $f=\bar f\circ \pi$, 
$$
S(\bar f(\bar u_1),\dots,\bar f(\bar u_n))
$$ 
holds, hence also 
$$
S(\bar f(\bar x_1),\dots,\bar f(\bar x_n))
$$ 
holds. 
\end{proof}

\begin{definition}\label{relationonto}
Let $f\colon\Ufrak\rightarrow\Wfrak$ be a morphism of structures over a language $\Lcal$. We will say that $f$ is \emph{relation-onto} if for every relation symbol $R$ of $\Lcal$ we have: for all $x_1,\dots,x_n\in\Ufrak$, if $\Wfrak$ satisfies $R(f(x_1),\dots,f(x_n))$ then there exist $u_1\sim_f x_1$, \dots, $u_n\sim_f x_n$ such that $\Ufrak$ satisfies $R(u_1,\dots,u_n)$.
\end{definition}

Note that the condition of being relation-onto does not need to be checked for the equality (as it is trivially satisfied). 

\begin{definition}
Given a morphism of $\Lcal$-structures $f\colon\Ufrak\rightarrow\Wfrak$, where $\Ufrak$ has base set $U$, the quotient $\Lcal$-structure $\Ufrak_f$ is defined as follows:
\begin{itemize}
\item the base set of $\Ufrak_f$ is $U_f$;
\item for each function symbol $h$ (including constant symbols), the interpretation of $h$ in $\Ufrak_f$ is given by: 
$$
h^{\Ufrak_f}(\bar x_1,\dots,\bar x_n)=h^{\Ufrak}(x_1,\dots,x_n);
$$
\item for each relation symbol $R$, the interpretation of $R$ in $\Ufrak_f$ is given by: $R^{\Ufrak_f}(\bar x_1,\dots,\bar x_n)$ holds if and only if there exist $u_1\in\bar x_1,\dots,u_n\in\bar x_n$ such that $R^{\Ufrak}(u_1,\dots,u_n)$ holds.
\end{itemize}
\end{definition}

\begin{proposition}\label{OjosAzules}
Let $f\colon\Ufrak\rightarrow\Wfrak$ be a morphism of $\Lcal$-structures. We have: 
\begin{enumerate}
\item The quotient structure $\Ufrak_f$ is indeed an $\Lcal$-structure.
\item The canonical map $\pi_f\colon\Ufrak\rightarrow\Ufrak_f$ is a $\Lcal$-morphism.
\item The induced map $\bar f\colon\Ufrak_f\rightarrow\Wfrak$ is an injective $\Lcal$-morphism.
\item The morphism $f$ is relation-onto if and only if $\bar f$ is a $\Lcal$-monomorphism.
\item The morphism $f$ is onto and relation-onto if and only if $\bar f$ is a $\Lcal$-isomorphism.
\end{enumerate}
\end{proposition}
\begin{proof}
The proof is easy and left to the reader (it comes from Lemma \ref{SetMon}). 
\end{proof}

The following lemma is well known.

\begin{lemma}\label{Lemoncito}
Let $\Ufrak$ be an $\Lcal$-structure, $\asymp$ a binary relation symbol, and $T_\asymp$ the theory of the equality for the symbol $\asymp$. The quotient structure $\Ufrak/\asymp^\Ufrak$ is an $\Lcal$-structure which satisfies $T_\asymp$ (hence it is a \emph{equalitarian structure}) and is elementarily equivalent to $\Ufrak$. 
\end{lemma}

\begin{proposition}\label{Lemon}
Let $\Lcal_0$ be a first order language. Let $\Ucal_0$ be a class of $\Lcal_0$-structures and $\Wfrak$ be an $\Lcal_0$-structure. Assume that for each structure $\Ufrak\in\Ucal_0$ there exists a morphism $f_\Ufrak\colon\Ufrak\rightarrow\Wfrak$ which is onto and relation-onto. Let $\Lcal_1$ be a language that contains $\Lcal_0$ and, given an interpretation for each symbol of $\Lcal_1\smallsetminus\Lcal_0$ in each structure $\Ufrak$ of $\Ucal_0$, we denote by $\Ufrak_1$ the new structure, and $\Ucal_1$ denotes the class of $\Lcal_1$-structures $\Ufrak_1$. If the collection of relations $\sim_{f_\Ufrak}$ is uniformly definable by an $\Lcal_1$-formula $\varphi(a,b)$ in $\Ucal_1$, then the algorithm $\Acal$ which does the following: 

{\tt \noindent In any $\Lcal_0$-sentence $F$, for each relation symbol $R$ (including the symbol of equality) that occurs in $F$, replace $R(x_1,\dots,x_n)$ by 
$$
\exists u_1,\dots,u_n\left(\bigwedge_{i=1}^n\varphi(u_i,y_i)\wedge R(u_1,\dots,u_n)\right);
$$}

\noindent uniformly encodes $(\Fcal_{\Lcal_0},\Wfrak)$ in $(\Fcal_{\Lcal_1},\Ucal_1)$. Moreover, 
\begin{itemize}
\item if the formula $\varphi(a,b)$ is existential then $\Acal$ uniformly encodes $(\Fcal_{\Lcal_0}^{\rm e},\Wfrak)$ in $(\Fcal_{\Lcal_1}^{\rm e},\Ucal_1)$; 
\item if the formula $\varphi(a,b)$ is positive existential then $\Acal$ uniformly encodes $(\Fcal_{\Lcal_0}^{\rm pe},\Wfrak)$ in $(\Fcal_{\Lcal_1}^{\rm pe},\Ucal_1)$.
\end{itemize}
\end{proposition}
\begin{proof}
Let us show that the algorithm $\Acal$ uniformly encodes $(\Fcal_{\Lcal_0},\Wfrak)$ in $(\Fcal_{\Lcal_1},\Ucal_1)$ (the same algorithm works analogously for the two other cases). By Proposition \ref{OjosAzules}, Item 5, for each $\Ufrak$ in the class $\Ucal_0$, $\Wfrak$ satisfies $F$ if and only if $\Ufrak_{f_\Ufrak}$ satisfies $F$. Let us define the $\Lcal_0$-structure $\Ufrak^{f_\Ufrak}$ by
\begin{itemize}
\item the base set of $\Ufrak^{f_\Ufrak}$ is the base set of $\Ufrak$;
\item function symbols are interpreted in $\Ufrak^{f_\Ufrak}$ as in $\Ufrak$;
\item for each relation symbol $R$ (including the equality) of $\Lcal_0$, 
$$
R^{\Ufrak^{f_\Ufrak}}(x_1,\dots,x_n)
$$ 
holds if and only if there exists $u_1,\dots,u_n\in\Ufrak$ such that $u_1\sim_{f_\Ufrak}x_1$, \dots, $u_n\sim_{f_\Ufrak}x_n$ and $R^\Ufrak(u_1,\dots,u_n)$ holds.
\end{itemize}
In particular, the symbol of equality is interpreted in $\Ufrak^{f_\Ufrak}$ as the relation $\sim_{f_\Ufrak}$. By Proposition \ref{OjosAzules}, Item 2, the $\Lcal_0$-structure $\Ufrak^{f_\Ufrak}$ satisfies the theory of equality $T_=$. By Lemma \ref{Lemoncito}, the structure $\Ufrak_{f_\Ufrak}$ satisfies $F$ if and only if $\Ufrak^{f_\Ufrak}$ satisfies $F$. Therefore, $\Ufrak_{f_\Ufrak}$ satisfies $F$ if and only if $\Ufrak_1$ satisfies $\Acal(F)$.
\end{proof}

\section{Case of integers}

\subsection{Some general uniform definitions in $\Ncal$ and $\Dcal$}

In this section we will show in particular that squaring powers of a prime is uniformly positive existentially definable in $\Ncal$ and in $\Dcal$ - see Notation \ref{cafe}, Items \ref{midp}, \ref{Dcal}, and \ref{Ncal}. 

When working with the structures $\Nfrak_p$, the string `$a\leq b$' stands for 
$$
\exists c(b=a+c).
$$

\begin{notation}
\begin{enumerate}
\item for each prime number $p$ we define 
$$
P_p^>=\{p^h\colon h\in\N\}\qquad\qquad P_p^\pm=\{\pm p^h\colon h\in\N\}
$$ 
$$
P_{p,0}^>=\{p^h\colon h\in\N_{>0}\}\qquad\qquad P_{p,0}^\pm=\{\pm p^h\colon h\in\N_{>0}\}
$$
\end{enumerate}
\end{notation}

\begin{lemma}\label{Tonia}
The formula $P(n)=R(1,n)$ uniformly positive existentially 
\begin{enumerate}
\item $\Lcal^{*,+}$-defines the collection of sets $P_p^>$ in $\Ncal$ (hence in particular $P_p^>$ is $\Lcal^{*,+}$-uped in $\Ncal$);
\item $\Lcal_T^*$-defines the collection of sets $P_p^\pm$ in $\Dcal$ (hence in particular $P_p^\pm$ is $\Lcal_T^*$-uped in $\Dcal$).
\end{enumerate}
\end{lemma} 
\begin{proof}
This comes immediately from the definitions. 
\end{proof}

\begin{lemma}\label{powers} 
The formula 
$$
P_0^\varepsilon(n)\colon
\begin{cases}
R(1,n)\wedge (n\geq2)&\textrm{if } \varepsilon \textrm{ is }>\\
R(1,n)\wedge T(n)&\textrm{if } \varepsilon \textrm{ is }\pm
\end{cases}
$$ 
uniformly positive existentially
\begin{enumerate}
\item $\Lcal^{*,+}$-defines the collection of sets $P_{p,0}^>$ in $\Ncal$ if $\varepsilon$ is $>$ (hence $P_{p,0}^>$ is $\Lcal^{*,+}$-uped in $\Ncal$).
\item $\Lcal_T^*$-defines the collection of sets $P_{p,0}^\pm$ in $\Dcal$ if $\varepsilon$ is $\pm$ (hence $P_{p,0}^\pm$ is $\Lcal_T^*$-uped in $\Dcal$).
\end{enumerate}
\end{lemma}
\begin{proof}
This comes immediately from the definitions.
\end{proof}

\begin{lemma}
Consider the positive existential formula 
$$
\bar\theta_P^\varepsilon(m,n)\colon P_0^\varepsilon(m)\wedge P_0^\varepsilon(n)\wedge R(m-1,n-m)
$$
over $\Lcal^{*,+}$ if $\varepsilon$ is $>$, and over $\Lcal_T^*$ if $\varepsilon$ is $\pm$. For each prime $p$, we have
\begin{enumerate}
\item $\Nfrak_p$ satisfies $\bar\theta_P^>$ if and only if $m,n\in P_{p,0}^>$ and $n=m^2$; and
\item $\Dfrak_p$ satisfies $\bar\theta_P^\pm$ if and only if $m,n\in P_{p,0}^\pm$
and 
\begin{itemize}
\item either $n=m^2$; or
\item $p=2$ and $(m,n)\in\{(-2,-8),(2,-2),(4,-2),(4,-8)\}$; or 
\item $p=3$ and $(m,n)=(3,-3)$.
\end{itemize}
\end{enumerate}
\end{lemma}
\begin{proof}
We leave to the reader the verification of the implications from the right to the left. Suppose that $\bar\theta_P^\varepsilon$ is satisfied in $\Dfrak_p$ or $\Nfrak_p$ (depending on $\varepsilon$). There exist integers $r,s,\ell$ such that $r>0$ and $s>0$ and there exist $\rho,\sigma,\lambda$ in $\{-1,1\}$ (or $=1$ if working in $\Nfrak_p$) so that 
$$
m=\rho p^r\qquad n=\sigma p^s\qquad n-m=\lambda p^\ell(m-1).
$$
By direct substitution we obtain
$$
\sigma p^s-\rho p^r=\lambda p^\ell(\rho p^r-1)
$$
and deduce 
$$
\sigma p^s+\lambda p^\ell=\rho p^r(\lambda p^\ell+1)
$$
which implies that $\ell$ is positive (looking at the latter equation modulo $p$). Write the above equation as
\begin{equation}\label{padic0}
\sigma p^s=\rho p^r-\lambda p^\ell+\rho\lambda p^{r+\ell}
\end{equation}
and consider it over the ring $\Z_p$ of $p$-adic integers (or simply as an equation written in base $p$). 

In the case of $\Nfrak_p$, Equation \eqref{padic0} gives 
\begin{equation}\label{nat}
p^s+p^\ell=p^r+p^{r+\ell}.
\end{equation}
Since the right-hand side has two non-zero $p$-adic digits (for some choice of digits containing $1$), we have either $s=r$ and $\ell=r+\ell$, or $s=r+\ell$ and $\ell=r$. But the former case is impossible since $r>0$. Hence $s=2r$ and we deduce that $n=m^2$.

Let us come back to the general case of integers. Note that by Equation \eqref{padic0}, if $\rho=\lambda$ then $\sigma=\lambda\rho=1$ and $s=2r$, hence $n=m^2$.

If $p\geq 3$ then, since the coefficients lie between $-1$ and $1$ and since $r+\ell>\max\{r,\ell\}$, we deduce, from the uniqueness of the $p$-adic expansion, choosing for example representative ``digits'' within 
$$
D=\left\{-\frac{p-1}{2},\dots,\frac{p-1}{2}\right\},
$$ 
that $r=\ell$. Therefore, we have 
\begin{equation}\label{padic}
\sigma p^s=(\rho-\lambda)p^r+\lambda\rho p^{2r}
\end{equation}
and if $\rho=\lambda$ then $\sigma=\lambda\rho=1$ and $s=2r$, hence $n=m^2$. If $\rho$ is distinct from $\lambda$ then $p$ must be $3$ since otherwise the right-hand side would have two non-zero digits while the left-hand side has only one. Equation \eqref{padic} becomes
$$
\sigma 3^s=2\rho 3^r-3^{2r}
$$
hence 
$$
\sigma 3^{s-r}=2\rho-3^r
$$
which implies $s=r$ (looking at the equation modulo $3$). Therefore we have
$$
\sigma=2\rho-3^r
$$
which can happen only if $\sigma=-1$, $\rho=1$ and $r=1$, hence $(m,n)=(3,-3)$. 

Suppose that $p=2$. Note that if $r=\ell$ and $\rho=\lambda$ then we conclude that $n=m^2$ as before. 

\noindent\textbf{Case $\rho=-\lambda$.} Equation \eqref{padic0} becomes
\begin{equation}\label{padic2}
2^{r+\ell}+\sigma 2^s=\rho(2^r+2^\ell).
\end{equation}
If $\sigma=1$ then $\rho=1$, which gives 
$$
2^{r+\ell}+2^s=2^r+2^\ell.
$$ 
Since $r+\ell>\max\{r,\ell\}$, by the uniqueness of the $2$-adic expansion with digits $\{0,1\}$, we have $r=\ell$, hence 
$$
2^{2r}+2^s=2^{r+1}.
$$ 
Hence $s=2r$ and $2^{2r+1}=2^{r+1}$, which is impossible since $r>0$.

Therefore, $\sigma=-1$. If $\rho=-1$ then Equation \eqref{padic2} becomes
$$
2^{r+\ell}+2^r+2^\ell=2^s,
$$
which gives $r=\ell$ (again by uniqueness), hence 
$$
2^{2r}+2^{r+1}=2^s
$$ 
and we deduce that $2r=r+1$, hence $r=1$ and $s=3$, which corresponds to the pair $(m,n)=(-2,-8)$. If $\rho=1$ then Equation \eqref{padic2} becomes
$$
2^{r+\ell}=2^s+2^r+2^\ell,
$$
which implies (again by uniqueness of the expansion) that either $s=r$, or $s=\ell$, or $r=\ell$. 
\begin{itemize}
\item If $s=r$ then $2^{r+\ell}=2^{r+1}+2^\ell$, hence $r+1=\ell$, hence $2^{2r+1}=2^{r+2}$ and $r=1$. This case corresponds to the pair $(m,n)=(2,-2)$. 
\item If $s=\ell$ then $2^{r+s}=2^{s+1}+2^r$, hence $s+1=r$, hence $2^{2s+1}=2^{s+2}$ and $s=1$. This case corresponds to the pair $(m,n)=(4,-2)$. 
\item If $r=\ell$ then $2^{2r}=2^s+2^{r+1}$, hence $s=r+1$, hence $2^{2r}=2^{r+2}$ and $r=2$. This case corresponds to the pair $(m,n)=(4,-8)$. 
\end{itemize}

\noindent\textbf{Case $\rho=\lambda$ and $r\ne\ell$.} Equation \eqref{padic0} becomes
\begin{equation}\label{padic3}
\sigma 2^s=\rho 2^r-\rho 2^\ell+2^{r+\ell}.
\end{equation}
If $\rho=1$ then 
$$
\sigma 2^s=2^r- 2^\ell+2^{r+\ell}>2^r>0,
$$ 
hence $\sigma=1$. Therefore, we have 
$$
2^s+2^\ell=2^r+2^{r+\ell},
$$ 
which we know from the analysis of Equation \eqref{nat} that it has no solution unless $\ell=r$. 
\end{proof}

\begin{corollary}\label{Elina}
There exists a positive existential formula 
\begin{enumerate}
\item $\theta_P^>(m,n)$ that uniformly $\Lcal^{*,+}$-defines the collection of sets $\{(p^h,p^{2h})\colon h\in\N\}$ in $\Ncal$ (hence squaring in $P_p^>$ is $\Lcal^{*,+}$-uped in $\Ncal$). 
\item $\theta_P^\pm(m,n)$ that uniformly $\Lcal_T^*$-defines the collection of sets $\{(\pm p^h,p^{2h})\colon h\in\N\}$ in $\Dcal$ (hence squaring in $P_p^\pm$ is $\Lcal_T^*$-uped in $\Dcal$). 
\end{enumerate}
\end{corollary} 
\begin{proof}
Choose
$$
\theta_P^>(m,n)\colon \bar\theta_P^>(m,n)\vee (m=1\wedge n=1),
$$
for Item 1 and 
$$
\theta_P^\pm(m,n)\colon ((m=1\vee m=-1)\wedge n=1)\vee(\bar\theta_P^\pm(m,n)\wedge n\ne-2\wedge n\ne-3\wedge n\ne -8)
$$
for Item 2.
\end{proof}

\begin{remark}
Corollary \ref{Elina} allows us to write in our formulas terms like $a^{2}$, $a^{4}$, $a^8$,\dots whenever $a$ is an element of $P_p$, $P_{p,0}$, $P_p^+$ or $P_{p,0}^+$. 
\end{remark}


\subsection{Multiplication uniformly in $\Ncal$ and $\Zcal$}
In this section we will first prove Item \ref{1N} of Theorem \ref{Teoremon1} and then deduce Item \ref{1Z} from it. 

\begin{lemma}\label{MP}
The collections of sets
$$
M_p=\{(n,p^a,np^a)\colon n\geq0\textrm{ and }a\geq 0\}
$$ 
are $\Lcal^{*,+}$-uped in $\Ncal$.
\end{lemma}
\begin{proof}
Following the strategy of the second author in \cite[Section 2]{PheidasPow}, we show that the following formula $\varphi(x,y,z)$
$$
P^>(y)\wedge z\geq x\wedge R(x,z)\wedge R(x+1,z+y)\wedge R(x+y,z+y^2)
$$
is true in $\Nfrak_p$ if and only if $(x,y,z)\in M_p$. 

If $z=xy$ and $y=p^a$ for some non-negative integer $a$, then we have 
\begin{itemize}
\item $z\ge x$; 
\item $z=xp^a$;
\item $z+y=xp^a+p^a=p^a(x+1)$; and 
\item $z+y^2=xp^a+p^{2a}=p^a(x+y)$,
\end{itemize}
hence $\Nfrak_p$ satisfies $\varphi(x,y,z)$. 

Suppose that $\Nfrak_p$ satisfies $\varphi(x,y,z)$. There exist integers $a,\alpha,\beta,\gamma$ such that $a\ge0$ and
\begin{equation}\label{ph0}
y=p^a
\end{equation}
\begin{equation}\label{ph1}
z=p^\alpha x
\end{equation}
\begin{equation}\label{ph2}
z+y=p^\beta(x+1)
\end{equation}
\begin{equation}\label{ph3}
z+y^2=p^\gamma(x+y).
\end{equation}
First note that if $x=0$ then $z=0$ and we are done, hence we suppose that $x$ is positive. From $z=p^\alpha x$, $z\ge x$ and $x\ge1$ we deduce that $\alpha$ is non-negative. Also we have $\beta$ and $\gamma$ non-negative (since from Equation \eqref{ph2} we have 
$$
x+1\le z+y=p^\beta(x+1)
$$ 
and from Equation \eqref{ph3} we have $x+y\le z+y^2=p^\gamma(x+y)$). 

From Equation \eqref{ph1}, \eqref{ph2} and \eqref{ph3} we obtain
\begin{equation}\label{ph4}
x(p^\alpha-p^\beta)=p^\beta-p^a
\end{equation}
and from Equation \eqref{ph0}, \eqref{ph1} and \eqref{ph3} we obtain
\begin{equation}\label{ph5}
x(p^\alpha-p^\gamma)=p^{a+\gamma}-p^{2a}.
\end{equation}

Let us prove that if two elements of $\{a,\alpha,\beta,\gamma\}$ are equal then $z=xy$. 
\begin{itemize}
\item If $a=\alpha$ then we conclude from Equations \eqref{ph0} and \eqref{ph1}.
\item If $a=\beta$ then we conclude that $\alpha=\beta$ from Equation \eqref{ph4} and $x>0$.
\item If $a=\gamma$ then we conclude that $\alpha=\gamma$ from Equation \eqref{ph5} and $x>0$.
\item If $\alpha=\beta$ then we conclude that $a=\beta$ from Equation \eqref{ph4} and $x>0$.
\item If $\alpha=\gamma$ then we conclude that $a=\gamma$ from Equation \eqref{ph5} and $x>0$.
\item If $\beta=\gamma$ then from Equations \eqref{ph4} and \eqref{ph5} we have 
$p^\beta-p^a=p^{a+\beta}-p^{2a}$, hence $p^\beta(1-p^a)=p^a(1-p^a)$, hence either $a=\beta$, in which case we can conclude as above, or $a=0$ and $\beta>0$. In the latter case, from Equation \eqref{ph4} we obtain $x(p^\alpha-p^\beta)=p^\beta-1>0$, hence $\alpha>\beta>0$, which is impossible since $p$ does not divide $p^\beta-1$. 
\end{itemize}

From now on, we may suppose that $a$, $\alpha$, $\beta$, and $\gamma$ are pairwise distinct. From Equation \eqref{ph4}, we have
\begin{equation}
\alpha>\beta\textrm{ if and only if }\beta>a
\end{equation}
hence either $\alpha>\beta>a$ or $\alpha<\beta<a$. Similarly, from Equation \eqref{ph5}, we have
\begin{equation}
\alpha>\gamma\textrm{ if and only if }\gamma>a
\end{equation}
hence either $\alpha>\gamma>a$ or $\alpha<\gamma<a$. So we have four possible orders:
\begin{enumerate}
\item $\alpha>\beta>\gamma>a$;\label{or1}
\item $\alpha>\gamma>\beta>a$;\label{or2}
\item $\alpha<\beta<\gamma<a$; or\label{or3}
\item $\alpha<\gamma<\beta<a$.\label{or4}
\end{enumerate}
From Equations \eqref{ph4} and \eqref{ph5} we have
\begin{equation}\label{ph6}
(p^\alpha-p^\gamma)(p^\beta-p^a)=(p^\alpha-p^\beta)(p^\gamma-p^a)p^a.
\end{equation}
Hence the orders \ref{or2} and \ref{or4} are impossible (otherwise the left-hand side would have smaller absolute value than the right-hand side). In case of order number \ref{or1}, the valuation at $p$ in Equation \eqref{ph6} gives 
$$
\gamma+a=\beta+2a,
$$ hence 
$$
\gamma=\beta+a>\gamma+a,
$$ 
which is absurd. In case of order number \ref{or3}, we obtain 
$$
\alpha+\beta=\alpha+\gamma+a,
$$ 
hence 
$$
\beta=\gamma+a>\beta+a,
$$ 
which is absurd.
\end{proof}

Next corollary proves Item \ref{1N} of Theorem \ref{Teoremon1}.

\begin{corollary}\label{MultN}
Multiplication is $\Lcal^{*,+}$-uped in $\Ncal$. 
\end{corollary}
\begin{proof}
The proof is identical to the proof of \cite[Lemma 3]{PheidasPow} using Lemma \ref{MP} instead of \cite[Lemma 2]{PheidasPow}. 
\end{proof}

Next corollary proves Item \ref{1Z} of Theorem \ref{Teoremon1}.

\begin{corollary}\label{upedinZ}
Multiplication is $\Lcal^*$-uped in $\Zcal$. 
\end{corollary}
\begin{proof}
Let $\mu(x,y,z)$ be a positive existential $\Lcal^{*,+}$-formula that uniformly defines multiplication $z=xy$ in $\Ncal$ (it exists from Corollary \ref{MultN}). Let $\bar\mu(x,y,z)$ be the $\Lcal^*$-formula obtained from $\mu$ by replacing (syntactically) all occurences of the form $\exists u$ (where $u$ is a variable) by $\exists u\geq0$. The (positive existential) $\Lcal^*$-formula 
$$
\mu_1(x,y,z)=\bar\mu(x,y,z)\wedge x\geq0\wedge y\geq0 \wedge z\geq0
$$
uniformly defines the set 
$$
\{(x,y,z)\colon z=xy\textrm{ and } x,y,z\geq0\}
$$ 
in $\Zcal$. 
The (positive existential) $\Lcal^*$-formula 
$$
\mu_2(x,y,z)=\bigvee_{\varepsilon\in\{-1,1\}^3}\varepsilon_1x\geq0\wedge
\varepsilon_2y\geq0\wedge\varepsilon_3z\geq0\wedge
\mu_1(\varepsilon_1x,\varepsilon_2y,\varepsilon_3z)
$$
uniformly defines $\{(x,y,z)\colon z=xy\}$ in $\Zcal$. 
\end{proof}

\subsection{Multiplication uniformly in $\Dcal$}

In this section we prove Item \ref{1D} of Theorem \ref{Teoremon1}.

\begin{lemma}\label{coprime} 
There is a positive existential $\Lcal_T^*$-formula $CO(x)$ that defines uniformly the collection of sets 
$$
CO_p=\{n\in\Z\colon p\nmid n\}
$$ 
in $\Dcal$ (hence the sets $CO_p$ are $\Lcal_T^*$-uped in $\Dcal$).
\end{lemma}
\begin{proof}
Consider the formula 
$$
\exists m(P_0^\pm(m)\wedge n|m-1).
$$ 
If $n\in CO_p$, then we can take $m=p^{\varphi(|n|)}$, since by Euler's theorem we know that $p^{\varphi(|n|)}$ is congruent to $1$ mod $n$. 

Conversely, if the formula is satisfied in some $\Dfrak_p$, then there exists $k\in\Z$ such that $nk=m-1$. Since $p$ divides $m$, it does not divide $n$.
\end{proof}

The next lemma defines squaring uniformly in each $CO_p$.
\begin{lemma}\label{Renia}
The collection of sets 
$$
\left\{(n,n^2)\colon n\in CO_p\right\}
$$ 
is $\Lcal_T^*$-uped in $\Dcal$. More precisely, for any integer prime $p$ we have: $n=m^2$ with $m,n\in CO_p$ if and only if $\Dfrak_p$ satisfies the following $\Lcal_T^*$-formula $\theta_{CO}(m,n)$
$$
CO(m)\wedge CO(n)\wedge\exists  a(P_0^\pm(a)\wedge m|a^2-1\wedge n|a^2 -1\wedge a^8-m\mid a^{16}-n).
$$
\end{lemma} 
\begin{proof}
If $n=m^2$ and $m,n\in CO_p$ we have two cases; if $|m|=n=1$ then any $a\in P_{p,0}^\pm$ works, and if $T(m)$ then $n>2$ and we can choose $a=p^{\phi(m)\phi(n)/2}$, where $\phi$ stands for Euler's function (recall that $\phi(n)$ is even for $n>2$).

Suppose that $\Dfrak_p$ satisfies $\theta_{CO}(m,n)$ for some $m,n\in CO_p$. Since $a\in P_{p,0}^\pm$ we have $a\geq2$. Since $m$ and $n$ divide $a^2-1$, we have $|m|<a^2$ and $|n|<a^2$. Since $a^8-m$ divides 
$$
a^{16}-n=a^{18}-m^2+m^2-n,
$$ 
we have:
\begin{enumerate}
\item $a^8-m$ divides $m^2-n$;
\item $|m^2-n|<a^4+a^2$ (since $|m|<a^2$ and $|n|<a^2$); and
\item $|a^8-m|>a^8-a^2$ (since $|m|<a^2$).
\end{enumerate}
By 1, we have that either $m^2-n=0$ or $|a^8-m|\leq |m^2-n|$. For the sake of contradiction, suppose that the latter is true. Then we have
$$
a^8-a^2<|a^8-m|\leq|m^2-n|<a^4+a^2
$$
hence, since $a\ge 2$ we get
$$
a^8<a^4+2a^2< a^4+a^4<a^8
$$
which is impossible. Therefore $m^2=n$.
\end{proof}




\begin{lemma}\label{CtimesP} 
The collection of sets 
$$
\{(x,y,z)\colon z=xy\textrm{ and }x\in CO_p\textrm{ and }y\in P_p^\pm\}
$$ 
is $\Lcal_T^*$-uped in $\Dcal$. More precisely, for any integer prime $p$, we have: $x=mn$ with $m\in CO_p$ and $n\in P_p^\pm$, if and only if $\Dfrak_p$ satisfies the formula $\rho_{CP}(m,n,x)$
$$
\begin{aligned}
& \hspace{80pt}(n=-1\wedge m=-x)\vee(n=1\wedge m=x)\vee\\
& \left(CO(m)\wedge  P_0^\pm(n)\wedge
\exists a,b (\theta_{CO}(m,a)\wedge\theta_P^\pm(n,b)\wedge\theta_{CO}(m+n,a+2x+b))\right).
\end{aligned}
$$
\end{lemma}
\begin{proof} Note that if $p$ does not divide $m$ and $n\in P_{p,0}^\pm$ then $p$ does not divide $m+n$, and note that $(m+n)^2=a+2mn+b$.
\end{proof}

We are now ready to show that squaring is $\Lcal_T^*$-uped in $\Dcal$.

\begin{lemma}\label{Renia2}
For any integer prime $p$ and for any $m,n\in \Z$ the following holds: $n=m^2$ if and only if $\Dfrak_p$ satisfies
$$
\begin{aligned}
&\exists a,b,u,v (P(a)\wedge P(b)\wedge CO(u)\wedge CO(v)\wedge\rho_{CP}(u,a,m)\wedge\\ 
&\hspace{52pt}\rho_{CP}(v,b,n)\wedge \theta_P^\pm(a,b)\wedge \theta_{CO}(u,v))
\end{aligned}
$$
\end{lemma} 
\begin{proof}
Choose 
$$
a=p^{\ord_p m}\qquad\textrm{and}\qquad u=\frac{m}{a}
$$ 
and do the same for $n$.
\end{proof}

It is standard to define multiplication using squaring: for any $m,n,h\in \Z$ the following holds:
$$
h=m\cdot n \textrm{ if and only if } (m+n)^2=m^2+2h+n^2.
$$

Hence multiplication is $\Lcal_T^*$-uped in $\Dcal$ and Item \ref{1D} of Theorem \ref{Teoremon1} follows.

\section{Pell equations uniformly}\label{SecPell}

If $x$ and $a$ are polynomials in $z$, we will denote by $x(a)$ the composition $x\circ a$.

Let us first remind some known facts about Pell equations. 

\begin{theorem}\label{Katerina}  
Let $F$ be a field of characteristic $p\ne 2$ and let $z$ be a variable. Let $a\in F[z]\smallsetminus F$. Any solution $(X,Y)=(x,y)$ in $F[z]$ of the equation
\begin{equation}\label{Pell}
X^2-(a^2-1)Y^2=1
\end{equation}
is of the form $(x,y)=(\pm x_n(a),y_n(a))$ where the pairs $(x_n(z),y_n(z))$ are defined by 
\begin{equation}
x_n(z)+\sqrt{z^2-1}y_n(z)=\left(z+\sqrt{z^2-1}\right)^n
\end{equation} 
by separating rational and irrational parts over $F(z)$.

Moreover, for any $m,n\in \Z$ we have 
\begin{enumerate}
\item $x_{m+n}(a)=x_m(a)x_n(a)-(a^2-1)y_m(a)y_n(a)$;
\item $y_{m+n}(a)=x_m(a)y_n(a)+x_n(a)y_m(a)$;
\item The integer $m$ divides in $\Z$ the integer $n$ if and only if the polynomial $y_m(a)$ divides in $F[z]$ the polynomial $y_n(a)$;
\item If $p\ne 0$ then for any $s\in \Z$ we have: $n= \pm mp^s$ if and only if $x_n(a)=x_m^{p^s}(a)$.
\item $y_n(a)$ is non constant if and only if $n\notin\{-1,0,1\}$.
\end{enumerate}
\end{theorem} 
\begin{proof} 
See \cite{PheidasZahidi1}. 
\end{proof}

Theorem \ref{Katerina} tells us essentially that the structure of the set of solutions of the Pell equation \eqref{Pell} does not depend on the parameter $a$, not only as a group, but also as a structure with the relation of divisibility and the function that takes $p^s$-th powers. 

\begin{notation} 
We consider the following two groups:
\begin{enumerate} 
\item $(\Z\times\mu_2,\oplus)$, where $\mu_2$ is the multiplicative group with two elements, has its law defined by
$$
(m,v)\oplus(n,w)=(wm+vn,vw).
$$
\item If $F$ is a field of characteristic $p\ne2$, then 
$$
\Sigma_a(F)\subseteq F[z]\times F[z]
$$ 
denotes the set of solutions of 
$$
X^2-(a^2-1)Y^2=1,
$$ 
where $a\in F[z]\smallsetminus F$. It is well known that the operation
$$
(x,y)\oplus(x',y')=(xx'-(a^2-1)yy',xy'+x'y)
$$
defines a group law on $\Sigma_a(F)$.
\end{enumerate}
\end{notation}

Let us define the class $\Qcal$ as the set 
$$
\{\Qfrak_p\colon p\textrm{ is prime}\}
$$ 
where the $\Lcal_T^*$-structures $\Qfrak_p$ are defined as follows
\begin{itemize}
\item the base set is $\Z\times\mu_2$;
\item $0$ is interpreted as $(0,1)$;
\item $1$ is interpreted as $(1,1)$;
\item $+$ is interpreted as $\oplus$;
\item $(u,v)\mid (x,y)$ is interpreted as ``$u$ divides $x$'';
\item $R((u,v),(x,y))$ is interpreted as ``$u\mid_p x$;
\item $T((u,v))$ is interpreted as ``$u$ is not in $\{-1,0,1\}$''.
\end{itemize}

Let $\beta$ be a unary predicate symbol interpreted in each $\Lcal_T^*$-structure $\Qfrak_p$ as 
\begin{center}
$\beta^{\Qfrak_p}((v_1,v_2))$ if and only if $v_2=1$.
\end{center}

\begin{lemma}\label{papa}
The symbol $\beta$ is $\Lcal_T^*$-uped in $\Qcal$ by 
$$
\zeta(v)\colon \exists w (v=w+w\vee v=w+w+1)
$$ 
\end{lemma}
\begin{proof}
The formula $\zeta(v)$ is satisfied in $\Qfrak_p$ if and only if there exist $w_1\in\Z$ and $w_2\in\mu_2$ such that
$$
(v_1,v_2)=2(w_1,w_2)=(2w_1w_2,1)
$$
or
$$
(v_1,v_2)=2(w_1,w_2)+(1,1)=(2w_1w_2+1,1), 
$$
and the latter happens if and only if $v_2=1$. 
\end{proof}

It is well known that $\Sigma_a(F)$ is isomorphic to the additive group $\Z\times\frac{\Z}{2\Z}$. We will use this fact in the following form:

\begin{lemma}\label{grupo}
The map
$$
\begin{array}{cccc}
\xi_{a,F}\colon&\Z\times\mu_2&\rightarrow&\Sigma_a(F)\\
&(n,\varepsilon)&\mapsto&(\varepsilon x_n,y_n).
\end{array}
$$
is an isomorphism of groups. 
\end{lemma}

\begin{notation}
Let $F$ be a field of characteristic $p\ne2$ and $a\in F[z]$ non-constant. We consider the following $\Lcal_T^*$-structure
$$
\Gfrak_a(F)=\left(\Sigma_a(F);(1,0),(a,1),\oplus,\mid,\tilde R,\tilde T\right)
$$
where
\begin{itemize}
	\item $(x,y)\mid(u,v)$ means ``$y$ divides $v$'';
	\item $\tilde R((x,y),(u,v))$ means ``there exists $s\in\Z$ such that $x^{p^s}=u$''
	\item $\tilde T(x,y)$ means ``$y$ is not a constant''.
\end{itemize}
\end{notation}

\begin{lemma}\label{NasimNb} For any field $F$ of characteristic $p\ne2$, and for each $a\in F[z]\smallsetminus F$, the $\Lcal_T^*$-structures $\Gfrak_a(F)$ and $\Qfrak_p$ are isomorphic through $\xi_{a,F}$.
\end{lemma}
\begin{proof} 
It is an immediate consequence of Theorem \ref{Katerina}.
\end{proof}

\begin{notation}
Consider the $\Lcal_A$-formula
$$
\delta(\alpha,v,w)\colon v^2+(\alpha^2-1)w^2=1
$$
and note that it is satisfied in $F[z]$ if and only if the pair $(v,w)$ is a solution of the Pell equation with parameter $\alpha$.
\end{notation}

\begin{lemma}\label{mantab} 
If $\alpha$ is interpreted as a non-constant element of $F[z]$ then the positive existential $\Lcal_A$-formula 
\begin{multline}
\eta(\alpha,v,w) \colon \delta(\alpha,v,w)\wedge(\exists x,y(\delta(\alpha,x,y)\wedge(v=x^2-(\alpha^2-1)y^2\wedge w=2xy)\vee \\
(v=(x^2-(\alpha^2-1)y^2)\alpha-(\alpha^2-1)2xy\wedge w=x^2-(\alpha^2-1)y^2+2\alpha xy))).\notag
\end{multline}
is satisfied in $F[z]$ if and only if $(u,v)$ is a solution of the Pell equation with parameter $\alpha$ and $u=x_n$ for some integer $n$.
\end{lemma}
\begin{proof} 
This is trivial from Lemmas \ref{papa} and then \ref{NasimNb}.
\end{proof}

\begin{lemma}\label{manta} 
For each $\varepsilon$ in $\{\pm1\}$, let us consider the following  positive existential $\Lcal_A$-formula $\Delta^{\varepsilon}(\alpha,x,x')$: 
$$
\begin{aligned}
&\eta(\varepsilon x,\varepsilon x')\wedge
\exists y,y',y_1,y_2(\delta(\alpha,\varepsilon x,y)\wedge\delta(\alpha,\varepsilon x',y')\wedge\\
&\hspace{40pt}\delta(\varepsilon x,\varepsilon x',y_1)\wedge\delta(\varepsilon x+1,\varepsilon x'+1,y_2))
\end{aligned}
$$
and write 
$$
\Delta(\alpha,x,x')\colon\bigvee_{\varepsilon\in\{\pm1\}}\Delta^{\varepsilon}(\alpha,x,x').
$$
Whenever $\alpha$ is assigned a non-constant element of $F[z]$, the formula $\Delta(\alpha,x,x')$ is satisfied in $F[z]$ if and only if $R_{F[z]}(x, x')$ holds (see Definition \ref{DefRA}).
\end{lemma}
\begin{proof} 
Note that $\Delta^{1}$ is analogous to the formula of Lemma 2.4 in \cite{PheidasZahidi1}.
\end{proof}

\section{The relation ``$y$ is a $p^{s}$-th power of $x$''}\label{pDivisibility}

In this section, we prove Theorems \ref{TeoremonA}, \ref{TeoremonB} and \ref{TeoremonC}. 

\begin{lemma}\label{viento}
Let $A$ be a commutative ring with unit. Let $x,y\in A$ be such that $y=x^{p^s}$ for some non-negative integer $s$. Write 
$$
x_n^2=(x-1+n)^{p^s+1},
$$ 
where $n=1,\dots,M$ for some integer $M\geq2$. We have 
$$
xy=x_1^2\qquad\textrm{ and }\qquad x+y=x_2^2-x_1^2-1.
$$
\end{lemma}
\begin{proof}
Let us show that the second equation holds. We have:
$$
\begin{aligned}
x+y&=x+x^{p^s}\\
&=x^{p^s+1}+x+x^{p^s}+1-x^{p^s+1}-1\\
&=(x+1)(x^{p^s}+1)-x^{p^s+1}-1\\
&=(x+1)(x+1)^{p^s}-x^{p^s+1}-1\\
&=(x+1)^{p^s+1}-x^{p^s+1}-1\\
&=x_2^2-x_1^2-1
\end{aligned}
$$
which proves the lemma.
\end{proof}

\begin{proof}[Proof of Theorem \ref{TeoremonA}]
Suppose that B\"uchi's problem has a positive answer for a triple $(A,C,M)$ and write $M_0=M_0(A,C)$, so that we have: \emph{any $M$-term B\"uchi sequence $(x_n)$ of $(A,C)$, with $M\geq M_0$, is of the form
$$
x_n^2=(f+n)^{p^s+1}
$$
for some non-negative integer $s$ and $f\in A$.}

Consider the following formulas from the language of rings
$$
\varphi_0(x_1,\dots,x_{M_0},x,y)\colon \Delta^{(2)}(x_{1}^2,\dots,x_{M_0}^2)=(2)\,\wedge 
xy=x_1^2\wedge x+y=x_2^2-x_1^2-1,
$$
$$
\varphi_{1}(x,y)\colon \exists x_1\dots\exists x_{M_0}\varphi_0(x_1,\dots,x_{M_0},x,y)
$$
and 
$$
\varphi_{M_0}(x,y)\colon\varphi_1(x,y)\vee\varphi_1(y,x).
$$

For short, we might write 
\begin{center}
``there exists $s\in\Z$ such that $y=x^{p^s}$''
\end{center} 
instead of `` there exists $s\in\N$ such that either $y=x^{p^s}$ or $x=y^{p^s}$''.

Let us prove Item \ref{Aimp} of Theorem \ref{TeoremonA}. Let $x,y\in A$ be such that $R_A(x,y)$ holds. By definition of $R_A$ we have $y=x^{p^s}$ for some integer $s$. If $s\ge0$, taking $x_n\in A$ such that $x_n^2=(x-1+n)^{p^s+1}$ the formula $\varphi_{M_0}(x,y)$ is satisfied in $A$ by Lemma \ref{viento}. Analogously if $s\le0$ then by taking  $x_n^2=(y-1+n)^{p^{-s}+1}$ the formula $\varphi_{1}(y,x)$ is true in $A$ by Lemma \ref{viento}.

Let us prove Item \ref{Bz} of Theorem \ref{TeoremonA} (note that one implication comes directly from Item 1). Let $x,y\in A$ be such that $A$ satisfies $\varphi_{M_0}(x,y)$ and $xy$ or $x+y$ is not in $C$. On the one hand, if $xy$ is not in $C$ then $x_1^2$ is not in $C$. On the other hand, if $x+y$ is not in $C$ then $x_2^2-x_1^2-1$ is not in $C$, hence one of $x_1^2$ and $x_2^2$ is not in $C$. Therefore, the sequence $(x_1,\dots,x_{M_0})$ is a B\"uchi sequence with at least one term non-constant and by hypothesis, there exists $f\in A$ such that 
$$
x_n^2=(f+n)^{p^s+1}
$$ 
for some non-negative integer $s$. Therefore we have a system of equations in $x$ and $y$
$$
\begin{cases}
xy=(f+1)^{p^s+1}\\
x+y=(f+2)^{p^s+1}-(f+1)^{p^s+1}-1
\end{cases}
$$
whose unique solutions are 
$$
(x,y)=(f+1,(f+1)^{p^s})\qquad\textrm{and}\qquad (x,y)=((f+1)^{p^s},f+1)
$$ 
(the verification is easy and is left to the reader). Hence either $y=x^{p^s}$ or $x=y^{p^s}$, i.e. $R_A(x,y)$ holds.
\end{proof}

\begin{proof}[Proof of Theorem \ref{TeoremonB}]
Within this proof, ``transcendental'' will always mean ``transcendental over $C$'', and ``algebraic'' will always mean ``algebraic over $C$''. 

The positive existential formula from the language 
$\Lcal_T=\{0,1,+,\cdot,T\}$  
$$
\begin{aligned}
&\varphi_\Ccal^T(x,y)\colon((T(xy)\vee T(x+y))\wedge\varphi_{M(\Ccal)}(x,y))\vee\\
&\hspace{50pt}\exists u\exists v((T(uv)\vee T(u+v))\wedge\varphi_{M(\Ccal)}(ux,vy)\wedge \varphi_{M(\Ccal)}(u,v))
\end{aligned}
$$
uniformly defines $R_A$ in $\Ccal$ over $\Lcal_T$. 

Indeed, if $R_A(x,y)$ holds then there exists an integer $s$ such that $y=x^{p^s}$. If either $xy$ or $x+y$ is transcendental then $A$ satisfies 
$$
(T(xy)\vee T(x+y))\wedge\varphi_{M(\Ccal)}(x,y),
$$ 
by Theorem \ref{TeoremonA}. If none of $xy$ and $x+y$ is transcendental then choose $u$ transcendental and $v=u^{p^s}$ if $s\geq0$, or choose $v$ transcendental and $u=v^{p^{-s}}$ if $s<0$. For these choices of $u$ and $v$, $A$ satisfies 
$$
(T(uv)\vee T(u+v))\wedge\varphi_{M(\Ccal)}(ux,vy)\wedge \varphi_{M(\Ccal)}(u,v).
$$ 

Suppose now that $A$ satisfies $\varphi_\Ccal^T(x,y)$. If $A$ satisfies 
$$
(T(xy)\vee T(x+y))\wedge\varphi_{M(\Ccal)}(x,y)
$$ 
then $R_A(x,y)$ holds by Theorem \ref{TeoremonA}. If not then there in particular both of $xy$ and $x+y$ are algebraic, hence both of $x$ and $y$ are algebraic. Also there exist $u,v\in A$ such that 
\begin{itemize}
\item $uv$ or $u+v$ is transcendental (hence $u$ or $v$ is transcendental);
\item there exists $r\in\Z$ such that $v=u^{p^r}$ (by Theorem \ref{TeoremonA} and the previous item); and
\item $A$ satisfies $\varphi_{M(\Ccal)}(ux,vy)$.
\end{itemize}
Note that the first and second items imply that both $u$ and $v$ are transcendental.

Suppose that $x$ or $y$ is not $0$ (otherwise $R_A(x,y)$ holds trivially).\\
\emph{Case 1:} If $uxvy$ or $ux+vy$ is transcendental then, by the third item and Theorem \ref{TeoremonA}, there exists $s\in\Z$ such that $vy=(ux)^{p^s}$, hence none of $x$ and $y$ is $0$ and 
$$
u^{p^r}y=(ux)^{p^s},
$$ 
which implies 
$$
yu^{p^r-p^s}=x^{p^s}
$$
and therefore, $r=s$ (since $u$ is transcendental but none of $x$ and $y$ is transcendental nor $0$). Hence $R_A(x,y)$ holds. 

\emph{Case 2:} If both $uxvy$ and $ux+vy$ are algebraic then both $ux$ and $vy$ are algebraic, hence they are $0$ (since $u$ and $v$ are transcendental but $x$ and $y$ are algebraic), which contradicts the fact that $x$ or $y$ is non-zero. 

This finishes the proof of Item \ref{BT} of Theorem \ref{TeoremonB}. 

Let us prove Item \ref{Bz}, namely, let us prove that the positive existential formula $\varphi_\Ccal^z(x,y)$ from the language $\Lcal_z=\{0,1,+,\cdot,z\}$
$$
\varphi_\Ccal^z(x,y)\colon \varphi_{M(\Ccal)}(x,y)\vee\exists u(\varphi_{M(\Ccal)}(z,u)\wedge(\varphi_{M(\Ccal)}(zx,uy)\vee\varphi_{M(\Ccal)}(ux,zy)))
$$
uniformly defines $R_A$ in $\Ccal$ over $\Lcal_z$. 

Suppose first that $R_A(x,y)$ holds. There exists an integer $s$ such that $y=x^{p^s}$. If $s\geq0$, then by taking $u=z^{p^s}$ the formula $\varphi_{M(\Ccal)}(zx,uy)$ holds in $A$ by Theorem \ref{TeoremonA}. If $s\leq0$, then by taking $u=z^{p^{-s}}$ the formula $\varphi_{M(\Ccal)}(ux,zy)$ holds in $A$ by Theorem \ref{TeoremonA}.

Suppose now that $A$ satisfies $\varphi_\Ccal^z(x,y)$. If $xy$ or $x+y$ is transcendental then as $A$ satisfies $\varphi_{M(\Ccal)}(x,y)$ we are done by Theorem \ref{TeoremonA}. So suppose that both $xy$ and $x+y$ are algebraic (hence both $x$ and $y$ are algebraic). 

Suppose that $A$ satisfies 
$$
\exists u(\varphi_{M(\Ccal)}(z,u)\wedge\varphi_{M(\Ccal)}(zx,uy))
$$ 
(the other case is done similarly). Since $A$ satisfies $\varphi_{M(\Ccal)}(z,u)$ 
and $z$ is transcendental (hence $zu$ or $z+u$ is transcendental), by Theorem \ref{TeoremonA}, there exists an integer $r$ such that
\begin{equation}\label{mouse}
u=z^{p^r}
\end{equation}
and in particular $u$ is transcendental. Note that if both $x$ and $y$ are $0$ then we are done. So we may assume that one of the two is non-zero. 

\emph{Case 1:} If $uy+zx$ or $uyzx$ is transcendental, as $A$ satisfies $\varphi_{M(\Ccal)}(zx,uy)$, there exists an integer $k$ such that $uy=(zx)^{p^k}$ by Theorem \ref{TeoremonA}. In particular, none of $x$ and $y$ is $0$. By Equation \eqref{mouse} we have $z^{p^r}y=(zx)^{p^k}$, hence
$$
z^{p^r-p^k}y=x^{p^k}
$$
which implies $k=r$ (since $x$ and $y$ are algebraic and non-zero and $z$ is transcendental) and the result follows. 

\emph{Case 2:} If both $uy+zx$ and $uyzx$ are algebraic then both $uy$ and $zx$ are algebraic, which is impossible since $u$ and $z$ are transcendental, $x$ and $y$ are algebraic, and at least one of $x$ or $y$ is non-zero. 
\end{proof}

\begin{proof}[Proof of Theorem \ref{TeoremonC}]
Consider the positive existential $\Lcal_T$-formulas
$$
\psi_1(x,y)\colon \exists x_1\dots\exists x_{M(\Ccal)}\left((T(x_1^2)\vee T(x_2^2-x_1^2-1))\wedge\varphi_0(x_1,\dots,x_{M(\Ccal)},x,y)\right).
$$
and
$$
\psi_\Ccal^T(x,y)\colon \psi_1(x,y)\vee\psi_1(y,x)
$$
where $\varphi_0$ is defined at the beginning of this section (note that we have replaced $M_0$ by $M(\Ccal)$ in the definition of $\varphi_0$). 

Let $x,y\in A$ be such $R_A^C(x,y)$ holds. By definition of $R_A^C$ we have $y=x^{p^s}$ for some integer $s$. As in the proof of Theorem \ref{TeoremonA}, if $s\ge0$, taking $x_n\in A$ such that 
$$
x_n^2=(x-1+n)^{p^s+1}
$$ 
the formula 
$$
\varphi_0(x_1,\dots,x_{M(\Ccal)},x,y)
$$ 
is satisfied in $A$ by Lemma \ref{viento}. Analogously if $s\le0$ then by taking  
$$
x_n^2=(y-1+n)^{p^{-s}+1}
$$ 
the formula 
$$
\varphi_0(x_1,\dots,x_{M(\Ccal)},y,x)
$$ 
is true in $A$ by Lemma \ref{viento}. Let us show that with these elections of the $x_n$, the structure $A$ satisfies 
$$
T(x_1^2)\vee T(x_2^2-x_1^2-1).
$$ 
Suppose that it is not the case (i.e. $x_1^2$ and $x_2^2-x_1^2-1$ are in $C$) and that $s\geq0$. By Lemma \ref{viento}, $xy=x_1^2$ and 
$$
x+y=x_2^2-x_1^2-1
$$ 
are in $C$, hence 
$$
x^2+y^2=(x+y)^2-2xy
$$ 
is in $C$ and we deduce that 
$$
(x-y)^2=x^2+y^2-2xy
$$ 
is in $C$. By hypothesis on $C$, it follows that $x-y\in C$, hence also 
$$
2x=(x+y)+(x-y)
$$ 
is in $C$ and 
$$
2y=(x+y)-(x-y)
$$ 
is in $C$. Therefore, again by hypothesis on $C$, $x$ and $y$ are in $C$, which gives a contradiction. Hence $A$ satisfies $\psi_1(x,y)$. Similarly, if $s\leq0$ then $A$ satisfies $\psi_1(y,x)$. Hence $A$ satisfies $\psi_\Ccal^T(x,y)$.

Suppose that $A$ satisfies $\psi_\Ccal^T(x,y)$. Since the sequence $(x_1,\dots,x_{M(\Ccal)})$ is a B\"uchi sequence with either $x_1^2$ or $x_2^2$ not in $C$, hence either $x_1$ or $x_2$ is not in $C$, and since $\Ccal$ is a B\"uchi class, there exists $f\in A$ such that 
$$
x_n^2=(f+n)^{p^s+1}
$$ 
for some non-negative integer $s$. We conclude as in the proof of Theorem \ref{TeoremonA}. 
\end{proof}

\section{Uniform encoding of the natural numbers}\label{unat}

In this section we will prove Theorem \ref{Teoremon2}. 

Let us call an algorithm \emph{identity algorithm} if it returns the input data. 

\begin{proof}[Proof of Item \ref{2Z} of Theorem \ref{Teoremon2}]
We want to prove that the pair $(\fpe_{\Lcal_A},\N)$ is uniformly encodable in the pair $(\fpe_{\Lcal^*},\Zcal)$. Following the strategy described in Section \ref{UnifEncod}, we will follow the ``One step encoding process'' for natural numbers, for the language $\Lcal=\Lcal^*$, the class $\Ucal=\Ncal$ and where the set $X$ is $\{\cdot\}$. 

The following algorithm $\Acal$ uniformly encodes $(\fpe_{\Lcal_A},\N)$ in $(\fpe_{\Lcal^*\cup X},\Zcal^X)$ (hence Item \ref{onestepa} of the process is fullfiled): given a sentence $F$ in $\fpe_{\Lcal_A}$, replace each occurrence in $F$ of $\exists x$ by $\exists x\geq0$ (or more formally ``$\exists x(\pos(x)\wedge$'' and taking care of where should close the parenthesis). It is clear that $F$ is true in $(\N;0,1,+,\cdot)$ if and only if $\Acal(F)$ is true in $(\Z;0,1,+,\cdot,\geq0,\mid_p)$.

By Corollary \ref{upedinZ}, multiplication is $\Lcal^*$-uped in the class $\Zcal$, hence also Item \ref{onestepb} is fulfilled. 
\end{proof}

\begin{proof}[Proof of Item \ref{2N} of Theorem \ref{Teoremon2}]
We want to prove that the pair $(\fpe_{\Lcal_A},\N)$ is uniformly encodable in the pair $(\fpe_{\Lcal^{*,+}},\Ncal)$. Following the strategy described in Section \ref{UnifEncod}, we will follow the ``One step encoding process'' for natural numbers, for the language $\Lcal=\Lcal^{*,+}$, the class $\Ucal=\Ncal$ and where the set $X$ is $\{\cdot\}$. Item \ref{onestepa} of the process comes trivially in this case: the required algorithm to uniformly encode $(\fpe_{\Lcal_A},\N)$ in 
$$
(\fpe_{\Lcal^{*,+}\cup X},\Ncal^X)
$$ 
is the identity algorithm, because a formula in $\fpe_{\Lcal_A}$ is true in $(\N;0,1,+,\cdot)$ if and only if it is true in $\Nfrak_p^X=(\N;0,1,+,\cdot,\mid_p)$. Item \ref{onestepb} asks for multiplication to be  $\Lcal^{*,+}$-uped in the class $\Ncal$, but this is Item \ref{1N} of Theorem \ref{Teoremon1}. 
\end{proof}

\begin{proof}[Proof of Item \ref{2O} of Theorem \ref{Teoremon2}]
We want to prove that the pair $(\fpe_{\Lcal_A},\N)$ is uniformly encodable in the pair 
$$
(\fpe_{\Lcal_{z,\ord,\ne}},\Omega)
$$ 
where $\Omega$ has been defined in Notation \ref{clason}. We will follow the ``two steps encoding process'' for natural numbers, where
\begin{itemize}
\item $\Lcal=\Lcal^*$ and $\Ucal=\Ncal$
\item $\Acal_0$ is the algorithm that uniformly encodes $(\fpe_{\Lcal_A},\N)$ in $(\fpe_{\Lcal^{*,+}},\Ncal)$
\item $\Lcal'=\Lcal_{z,\ord,\ne}$ and $\Vcal=\Omega$
\item $\Ucal_{\rm ind}=\{\Nfrak_p\colon p\in\Pcal\}$, where $\Pcal$ is the set of primes for which there exists at least one structure in $\Omega$ of characteristic $p$.
\item $\Omega$ is partitionned into subclasses $\Omega_{\Nfrak_p}$ where $\Omega_{\Nfrak_p}$ is the class of all structures in $\Omega$ of characteristic $p$.
\item $Y$ is the set of symbols $\{R,S\}$
\item $R$ is interpreted in each $\Lcal_{z,\ord,\ne}$-structure $\Ufrak$ of $\Omega$ by the relation $R_{\Ufrak}$ defined by ``$R_{\Ufrak}(x,y)$ if and only if there exists an integer $s$ such that $y=x^{p^s}$'', where $p$ is the characteristic of $\Ufrak$
\item $S(x,y)$ is interpreted in each $\Ufrak$ in $\Omega$ as ``$\ord_\pfrak(x)=\ord_\pfrak(y)$'' (recalling that to each structure in $\Omega$ is associated exactly one local parameter $z$ at a prime divisor $\pfrak$ - see Notation \ref{clason})
\end{itemize}

We will apply Proposition \ref{Lemon} for each prime number in $\Pcal$. Fix such a prime $p$. Following the notation of Proposition \ref{Lemon}, consider:
\begin{itemize}
\item $\Lcal_0=\{1,z,\cdot,R\}$
\item $\Lcal_1=\Lcal_0\cup\{S\}$
\item $\Ucal_0^p=\Omega_{\Nfrak_p}^{*,{\rm reg}}$ is the class of $\Lcal_0$-structures with base set $\{f\in U\colon f\ne0\textrm{ and } \ord_\pfrak(f)\geq0\}$ as $U$ ranges among the base sets of structures of characteristic $p$ in $\Omega$, and where $R$ is interpreted as $R_{\Ufrak}$ restricted to $U\smallsetminus\{0\}$. Note that there can be more than one structure in $\Ucal_0$ with a given base set (depending on the choice of the local parameter). 
\item $\Ucal_1^p$ is the class of $\Lcal_1$-structures obtained from $\Ucal_0^p$ when interpreting $S(x,y)$ by ``$\ord_\pfrak(x)=\ord_\pfrak(y)$''. 
\item $\Wfrak^p$ is the $(\Nfrak_p,j)$-induced $\Lcal_0$-structure, where $j\colon\Lcal^{*,+}\rightarrow\Lcal_0$ is the bijection
$$
\begin{array}{|c||c|c|c|c|c|}\hline
 \alpha&0&1&+&R\\\hline
j(\alpha)&1&z&\cdot&R\\\hline
\end{array}
$$
\item For each $\Ufrak$ in $\Ucal_0^p$, let $f_\Ufrak\colon \Ufrak\rightarrow\Wfrak^p$ be the map that sends $x\in\Ufrak$ to its order at $\pfrak$. 
\end{itemize}

Let us prove that the hypothesis of Proposition \ref{Lemon} are satisfied. First, $f_\Ufrak$ is trivially a morphism for each $\Ufrak$ in $\Ucal_0^p$, and it is onto because in each structure in $\Ucal_0^p$ we have only regular functions. 

Let us prove that $f_\Ufrak$ is relation-onto. Let $x_1,x_2\in\Ufrak$ be such that 
$$
\ord_\pfrak(x_1)\mid_p \ord_\pfrak(x_2).
$$ 
Taking 
$$
u_1=z^{\ord_\pfrak(x_1)}\qquad\textrm{and}\qquad u_2=z^{\ord_\pfrak(x_2)},
$$ 
we have $u_1\sim_{f_\Ufrak}x_1$ and $u_2\sim_{f_\Ufrak}x_2$ and $R^\Ufrak(u_1,u_2)$ holds.

The collection of relations $\sim_{f_\Ufrak}$ is trivially $\Lcal_1$-uped in $\Ucal_1^p$. Therefore, we can apply Proposition \ref{Lemon} to obtain the algorithm $\Acal_p=\Acal$ that actually does not depend on $p$. 

At this point we have the algorithms $\Acal_0$, $\Acal_{\Lcal^{*,+}}^{\Lcal^0}$ and $\Acal$. 

Let $\Acal'$ be the algorithm that transforms an $\Lcal_1$-sentence into a sentence over the language $\Lcal_{z,\ord,\ne}\cup Y$ by replacing each occurrence of the form $Qx$, where $Q$ is a quantifier, by $Qx\ne0$. It is clear that for each $\Lcal_1$-sentence $F$ and for each $\Ufrak$ in the class $\Ucal_1^p$, the corresponding structure in $\Omega_{\Nfrak_p}^{Y}$ satisfies $\Acal'(F)$ if and only if $\Ufrak$ satisfies $F$. 

In order to conclude, it is now enough to give uniform positive existential $\Lcal_{z,\ord,\ne}$-definitions of the elements of $Y=\{R,S\}$ in the class $\Omega$. 

For the symbol $R$, this is Item \ref{Bz} of Theorem \ref{TeoremonB}. The positive existential formula 
$$
\exists u,v(x=uy\wedge y=vx\wedge\ord(u)\wedge\ord(v))
$$
uniformly $\Lcal_{z,\ord,\ne}$-defines $S$ in the class $\Omega$. 

Schematically, following the ``two steps encoding process'', we performed:
$$
(\fpe_{\Lcal_A},\N)\xrightarrow{\Acal_0}(\fpe_{\Lcal^{*,+}},\Ncal)\supseteq (\fpe_{\Lcal^{*,+}},\{\Nfrak_p\colon p\in\Pcal\})
\xrightarrow{\Acal_{\Lcal^{*,+}}^{\Lcal_0}}(\fpe_{\Lcal_0},\left\{\Wfrak^p\colon p\in\Pcal\right\})
$$
$$
\xrightarrow{\,\Acal\,}\left(\fpe_{\Lcal_1},\bigcup_{p\in\Pcal}\Omega_{\Nfrak_p}^{*,\rm reg}\right)\xrightarrow{\,\Acal'\,}
\left(\fpe_{\Lcal_{z,\ord,\ne}\cup Y},\bigcup_{p\in\Pcal}\Omega_{\Nfrak_p}^{Y}\right)
\xrightarrow{\Acal_{Y}^{R,S}}(\fpe_{\Lcal_{z,\ord,\ne}},\Omega)
$$
\end{proof}

\begin{proof}[Proof of Item \ref{2Ob} of Theorem \ref{Teoremon2}]
This comes immediately from Proposition \ref{Atransform} and Item \ref{2Ob} of Theorem \ref{Teoremon2}.
\end{proof}

\section{Uniform encoding of $\Z$ in the language $\Lcal_T$}\label{uint}

In this section we will prove Theorem \ref{Teoremon2b}.

\begin{proof}[Proof of Item \ref{2bZ} of Theorem \ref{Teoremon2b}]
We want to prove that $(\fpe_{\Lcal_A},\Z)$ is uniformly encodable in $(\fpe_{\Lcal^*},\Zcal)$. We proceed as in the proof of Item \ref{2N} of Theorem \ref{Teoremon2}  (following the ``one step encoding process'') with $\Lcal=\Lcal^*$, $\Ucal=\Zcal$ and $X=\{\cdot\}$. Multiplication is $\Lcal^*$-uped in the class $\Zcal$ by Item \ref{1D} of Theorem \ref{Teoremon1}. 
\end{proof}

\begin{proof}[Proof of Item \ref{2bD} of Theorem \ref{Teoremon2b}]
We prove that $(\fpe_{\Lcal_A},\Z)$ is uniformly encodable in $(\fpe_{\Lcal_T^*},\Dcal)$.
Follow the ``one step encoding process'') with $\Lcal=\Lcal_T^*$, $\Ucal=\Dcal$ and $X=\{\cdot\}$. Multiplication is  $\Lcal_T^*$-uped in the class $\Dcal$ by Item \ref{1Z} of Theorem \ref{Teoremon1}. 
\end{proof}

\begin{proof}[Proof of Item \ref{2bCT} of Theorem \ref{Teoremon2b}]
We prove that $(\fpe_{\Lcal_A},\Z)$ is uniformly encodable in $(\fpe_{\Lcal_T},\Ccal)$, where $\Ccal$ is the class of all polynomial rings over a field of odd positive characteristic. 

We will first follow the ``two steps encoding process'' to uniformly encode $(\fpe_{\Lcal_A},\Z)$ in $(\fpe_{\Lcal_T^*},\Qcal)$, with
\begin{itemize}
\item $\Lcal=\Lcal_T^*$ and $\Ucal=\Dcal$
\item $\Acal_0$ is the algorithm that uniformly encodes $(\fpe_{\Lcal_A},\Z)$ in $(\fpe_{\Lcal_T^*},\Dcal)$
\item $\Lcal'=\Lcal=\Lcal_T^*$ and $\Vcal=\Qcal$
\item $\Ucal_{\rm ind}=\Ucal$
\item $\Qcal$ is partitionned into one-element subsets $\{\Qfrak_p\}$
\item $Y$ has only one symbol $\beta$
\item $\beta$ is interpreted in each $\Lcal_T^*$-structure $\Qfrak_p$ of $\Qcal$ as ``$\beta((u,v))$ if and only if $v=1$''.
\end{itemize}

Note that $\beta$ is $\Lcal_T^*$-uped in $\Qcal$ by Lemma \ref{papa}. In this context $\Bcal'$ is the algorithm that transforms each occurrence of $Qx$ in an $\Lcal_T^*$-sentence by $Qx(\beta x)\wedge$ (with the usual abuse of notation). Schematically, we have:
$$
(\fpe_{\Lcal_A},\Z)\xrightarrow{\,\,\,\Acal_0\,\,\,}(\fpe_{\Lcal},\Dcal)= (\fpe_{\Lcal},\Dcal_{\rm ind})
$$
$$
\xrightarrow{\,\,\,\Bcal'\,\,\,}
\left(\fpe_{\Lcal\cup Y},\bigcup_{p}\{\Qfrak_p\}^{Y}\right)
\xrightarrow{\Acal_Y^{\beta}}(\Acal_Y^{\beta}(\fpe_{\Lcal\cup Y}),\Qcal)
\subseteq(\fpe_{\Lcal},\Qcal).
$$

We will now use Proposition \ref{Quentin} with 
\begin{itemize}
\item $\Gcal=\fpe_{\Lcal_A}$ and $\Mfrak=\Z$
\item $\Gcal'=\fpe_{\Lcal_T^*}$ and $\Ucal=\Qcal$
\item $\Gcal''=\fpe_{\Lcal_T}$ and $\Vcal=\Ccal$
\item $\Ucal_{\rm ind}=\Ucal$
\item $\Ccal$ is partitioned into subclasses $\Ccal_{\Qfrak_p}$, where $\Ccal_{\Qfrak_p}$ is the class of all polynomial rings of characteristic $p$. 
\end{itemize}
In order to conclude we need to find an algorithm $\Bcal$ such that for each $\Qfrak_p\in\Qcal$, the pair $(\fpe_{\Lcal_T^*},\Qfrak_p)$ is uniformly encodable in $(\fpe_{\Lcal_T},\Ccal_{\Qfrak_p})$. Fix a prime $p$ and let 
$$
F=Q_1u_1\dots Q_nu_n G(u_1,\dots,u_n)
$$
be an $\Lcal_T^*$-sentence in normal prenex form (in our case we need only to consider existential quantifiers, but the whole proof actually goes through when universal quantifiers are allowed). 

Write $G^1,\dots,G^m$ the atomic formulas that appear in $F$ and let us describe the algorithm (following the syntax as in Cori and Lascar \cite{CoriLascar}):
\begin{enumerate}
\item \textbf{Term by term substitution of constants and function symbols.} In each term of $G$, formally replace (in the following order)
	\begin{enumerate}
	\item each occurrence of $0$ by $(1,0)$;
	\item each occurrence of $1$ by $(\alpha,1)$, for some new (fixed) variable $\alpha$; 
	\item each $u_i$ by $(v_i,w_i)$, for some new (fixed) variables $v_i$ and $w_i$;
	\item each string of the form $(x,y)+(x',y')$ by
	$$
	(xx'-(\alpha^2-1)yy',xy'+x'y)
	$$ 
	until the whole term becomes a single pair. 
	\end{enumerate}
Call $G_0$ the word resulting from $G$, and $G_0^1,\dots,G_0^m$ the words resulting from the corresponding atomic formulas $G^1,\dots,G^m$.
\item \textbf{Substitution of the relation symbols: first component}:
	\begin{enumerate}
	\item In $G_0$, delete any of the $G_i^1$ (and its corresponding connective if any) where appears $\mid$ or $T$.
	\item Replace each pair by its first component.
	\item Replace each $R(x,x')$ by the formula $\Delta(\alpha,x,x')$ from Lemma \ref{manta}
	and write 
	$$
	G_1(\alpha,u_1,\dots,u_n,v_1,\dots,v_n)
	$$ 
	the resulting $\Lcal_T$-formula.
	\end{enumerate}
\item \textbf{Substitution of the relation symbols: second component}:
	\begin{enumerate}
	\item In $G_0$, delete any of the $G_i^1$ (and its corresponding connective if any) 	where appears $R$.
	\item Replace each pair by its second component.
	\item Replace each $x\mid y$ by $\exists t (y=tx)$ (and don't do anything to $T$)
	and write 
	$$
	G_2(\alpha,u_1,\dots,u_n,v_1,\dots,v_n)
	$$
	the resulting $\Lcal_T$-formula.
	\end{enumerate}
\item Define $\Bcal(F)$ as 
$$
Q_1u_1Q_1v_1\dots Q_nu_n Q_nv_n\exists\alpha \left(T(\alpha)\wedge G_1\wedge G_2\wedge \bigwedge_{i=1}^n\delta(\alpha,u_i,v_i)\right)
$$ 
\end{enumerate}

Observe that one of $G_1$ and $G_2$ must be non-empty: if no relations except equality appear in $G$, then we did not delete anything, and if a relation that is not equality appears in $G$ then it can be deleted only in one of $G_1$ or $G_2$. 

The algorithm $\Bcal$ works thanks to Lemma \ref{NasimNb}.

\end{proof}

\end{document}